\documentclass[11pt]{article}%
\usepackage{algorithm}
\usepackage{algpseudocode}
\usepackage{amssymb}
\usepackage{amsfonts}
\usepackage{amsmath}
\usepackage{graphicx}%
\providecommand{\U}[1]{\protect\rule{.1in}{.1in}}
\usepackage{tikz}
\usetikzlibrary{shapes}
\usepackage{pgfplots}
\usepackage{verbatim}
\usepgfplotslibrary{statistics}
\pgfplotsset{compat=1.8}
\newcommand{\quotes}[1]{``#1''}
\usepackage{multicol}
\usepackage{amsmath,amssymb,amsthm}
 
\renewcommand\and{\end{tabular}\kern-\tabcolsep\ and\ \kern-\tabcolsep\begin{tabular}[t]{c}}
\let\origthanks\thanks
\renewcommand\thanks[1]{\begingroup\let\rlap\relax\origthanks{#1}\endgroup}

\newtheorem{assumption}{Assumption}
\newtheorem{theorem}{Theorem}

\newtheorem{corollary}{Corollary}

\newtheorem{definition}{Definition}

\newtheorem{lemma}{Lemma}

\newtheorem{proposition}{Proposition}

\usepackage{tikz}
\usetikzlibrary{positioning}

\begin{document}

\title{The Viral State Dynamics of the Discrete-Time NIMFA Epidemic Model}
\author{Bastian Prasse\thanks{ Faculty of Electrical Engineering, Mathematics and
Computer Science, P.O Box 5031, 2600 GA Delft, The Netherlands; \emph{email}:
b.prasse@tudelft.nl, p.f.a.vanmieghem@tudelft.nl} \and Piet Van Mieghem\footnotemark[1]}
\date{Delft University of Technology \\
March 19, 2019}
\maketitle
\begin{abstract}
The majority of research on epidemics relies on models which are formulated in continuous-time. However, real-world epidemic data is gathered and processed in a digital manner, which is more accurately described by discrete-time epidemic models. We analyse the discrete-time NIMFA epidemic model on directed networks with heterogeneous spreading parameters. In particular, we show that the viral state is increasing and does not overshoot the steady-state, the steady-state is globally exponentially stable, and we provide linear systems that bound the viral state evolution. Thus, the discrete-time NIMFA model succeeds to capture the qualitative behaviour of a viral spread and provides a powerful means to study real-world epidemics. 
\end{abstract}

\section{Introduction}\label{sec:introduction}
Originating from the study of infectious human diseases \cite{bailey1975mathematical, anderson1992infectious}, epidemiology has evolved into a field with a broad spectrum of applications, such as the spread of computer viruses, opinions, or social media content \cite{pastor2015epidemic, nowzari2016analysis}. The mutual characteristic of epidemic phenomena is that they can be modelled by a viral infection, i.e. every individual is either infected (with the opinion, social media content, etc.) or healthy. An imperative element for epidemics is the infection of one individual by another, provided that the individuals are linked (for instance by physical proximity). The epidemic model that we consider in this work describes the spread of a virus on a higher level, by merging individuals with similar characteristics (such as residence or age) into groups.

We consider a network of $N$ nodes\footnote{In this work, we use the words group and node interchangeably.}, and we denote the fraction of infected individuals of group $i$ by the viral state $v_i(t)$ at any time $t\ge 0$. For node $i$, the \textit{continuous-time} NIMFA model \cite{van2009virus, van2011n, mieghem2014homogeneous} with heterogeneous spreading parameters is given by 
\begin{align}
	\frac{d v_i (t)}{d t } & = - \delta_i v_i(t) +\left( 1 - v_i(t)\right) \sum^N_{j = 1} \beta_{i j} v_j(t).  \label{NIMFA_continuous}
\end{align}
Here, $\beta_{ij} \ge 0$ denotes the infection rate from group $j$ to group $i$, and $\delta_i>0$ denotes the curing rate of group $i$. If $\beta_{ij} > 0$, then infections occur from group $j$ to group $i$, and we emphasise that $\beta_{ii} = 0$ does \textit{not} hold in general, since members of the same group $i$ may possibly infect one another. The \textit{discrete-time NIMFA model} is obtained from the continuous-time NIMFA (\ref{NIMFA_continuous}) by applying Euler's method \cite{stoer2013introduction}, with sampling time $T > 0$, and the discrete-time curing and infection rates follow as $q_i = \delta_i T$ and $w_{ij} = \beta_{ij} T$, respectively. 
\begin{definition}[Discrete-Time NIMFA Model]
 The discrete-time NIMFA model is given by
\begin{align}
v_i [k + 1] & = (1 - q_i) v_i[k] +  (1 - v_i[k])  \sum^N_{j=1} w_{i j} v_j[k] \label{NIMFA_disc_}
\end{align}
for every group $i=1, ..., N$, where $k \in \mathbb{N}$ denotes the discrete time slot, $q_i >0$ is the discrete-time \emph{curing rate}, and $w_{i j} \ge 0$ is the discrete-time \emph{infection rate} from group $j$ to group $i$. 
\end{definition}
As vector equations, (\ref{NIMFA_disc_}) reads
\begin{align}
v [k + 1] & = \textup{\textrm{diag}}(u - q)  v[k] + \textup{\textrm{diag}}(u - v[k]) W v[k],\label{NIMFA_disc_stacked}
\end{align}
where the viral state vector at discrete time $k$ equals $v[k] = (v_1 [k], ..., v_N [k])^T$, the curing rate vector equals $q = (q_1, ..., q_N)^T$, the $N \times N$ infection rate matrix $W$ is composed of the elements $w_{i j}$, and $u$ is the $N \times 1$ all-one vector. The infection rate matrix $W$ corresponds to a weighted adjacency matrix: If $w_{ij}>0$, then there is a directed link from node $j$ to node $i$. The \textit{steady-state vector} $v_\infty$ of the discrete-time NIMFA model (\ref{NIMFA_disc_stacked}) is significant, because it corresponds to the endemic state of the disease in the network.
\begin{definition}[Steady-State Vector]
The steady-state vector $v_\infty = (v_{\infty, 1}, ..., v_{\infty, N})^T$ is, if existent, the non-zero equilibrium of the discrete-time NIMFA model (\ref{NIMFA_disc_}), which satisfies
  \begin{align} \label{steady_state_componentwise} 
  \sum^N_{j=1} w_{ij}  v_{\infty, j} =  q_i \frac{v_{\infty, i}}{1 - v_{\infty, i}}, \quad i=1, ..., N.
  \end{align}  
\end{definition}
We argue that \textit{the discrete-time NIMFA system (\ref{NIMFA_disc_stacked}) is (one of) the simplest epidemic models that meets the practical requirements of modelling real-world epidemics on networks.} In particular, the NIMFA system succeeds to exhibit the following six properties, which are crucial for modelling and processing real-world epidemic data:    
\begin{enumerate}
\item[P1.] The viral state $v_i$ of every node $i$ corresponds to a fraction of infected individuals in group $i$, and, hence, $0\le v_i \le 1$ for every node $i$. In theory, modelling an epidemic per individual, e.g. by assigning every individual of the population a binary value (healthy or infected), may be more accurate than combining individuals into groups. However, it is infeasible in practice to determine the viral state of \textit{every individual at every time $t$}. Instead, one only has access to a (ideally unbiased) sampled subset of individuals. These samples give an estimate of the fraction of infected individuals in a homogeneous group, which may be a contiguous geographic region or a set of individuals with the same characteristics (same age, gender, vaccination status etc.).

\item[P2.] The viral state $v$ evolves in discrete time $k$. For the simulation of a viral spread, an implicit discretisation is performed for the majority of continuous-time epidemic models due to the absence of closed-form solutions for the viral state $v(t)$. Hence, a more accurate approach is to directly study the epidemic model in discrete-time. Furthermore, data on real-world epidemics is often collected periodically\footnote{For instance, the German Robert Koch Institute gathers and provides online access to cases of notifiable diseases with the web-based query system \textit{SurvStat@RKI 2.0} on a weekly basis.}, and discrete-time models circumvent the incomplete knowledge of the viral state of time spans between two measurements.

\item[P3.] The greater the viral state $v_j[k]$ of each neighbour $j$ of a node $i$ at time $k$, the greater is the viral state of node $i$ at time $k+1$. More specifically, the NIMFA model (\ref{NIMFA_disc_}) accounts for the infection from group $j$ to group $i$ by the term 
\begin{align} \label{kjbkjbasss}
w_{i j} (1 - v_i[k]) v_j[k],
\end{align}
which is proportional to both the fraction $(1 - v_i[k])$ of healthy individuals of group $i$ at time $k$ and the fraction $v_j[k]$ of infected individuals of group $j$ at time $k$. The infection rate $w_{i j}$ varies for any pair of groups $i,j$, which takes the heterogeneity of the contact between groups into account (for instance, group $i$ and group $j$ could be two geographical regions that are either adjacent or far apart). The complete infection term of group $i$ in the NIMFA model (\ref{NIMFA_disc_}) follows naturally by linear superposition of the infection terms (\ref{kjbkjbasss}) for every group $j$.

\item[P4.] There is a curing term that opposes the infection of node $i$ by its neighbours. In particular, the curing term $(1-q_i)v_i[k]$ of group $i$ in the NIMFA model (\ref{NIMFA_disc_}) is proportional to the fraction $v_i[k]$ of infected individuals of group $i$. The curing rate $q_i$ varies for different groups $i$, which accounts for heterogeneous capacities of the groups to heal from the virus (for instance, a group may refer to either younger or older individuals).

\item[P5.] There is a unique \cite{fall2007epidemiological, pare2018analysis} non-zero equilibrium $v_\infty$, which corresponds to the endemic state of the virus. Furthermore, if the disease does not die out, then the viral state $v$ approaches the endemic viral state $v_{\infty}$, i.e. $v[k] \rightarrow v_\infty$ for $k \rightarrow \infty$, which we show in this work. To the best of our knowledge, the convergence of the viral state $v(t)$ to the steady-state $v_\infty$ has only been shown \cite{ahn2013global, khanafer2016stability} for the \textit{continuous-time} NIMFA model (\ref{NIMFA_continuous}).

\item[P6.] The viral state is increasing, i.e. $v_i[k] > v_i[k-1]$ for any node $i$ at any time $k$, provided that the initial viral state $v[1]$ is close to zero (almost disease-free), which we show in this work. The viral state $v_i$ of node $i$ typically refers to \textit{cumulative} variables in practical applications, which are increasing and close to zero at the beginning of the outbreak of the disease. For instance \cite{pare2018analysis}, the viral state $v_i[k]$ of node $i$ may refer to the total number of deaths by cholera of group $i$ up to time $k$.

\end{enumerate}
For real-world applications, the usefulness of an epidemic model does not reduce to solely satisfying the properties P1--P6. An epidemic model must additionally be capable of giving answers to questions which are relevant to practical use-cases. In particular, we identify three questions.
 \begin{enumerate}
 \item[Q1.] \textit{How to fit the NIMFA model (\ref{NIMFA_disc_}) to real-world data?} In applications, we do not (exactly) know the infection rate matrix $W$ or the curing rates $q$. In a recent work \cite{prasse2018networkreconstruction}, we derived an efficient method to learn the spreading parameters $W, q$ of the NIMFA model (\ref{NIMFA_disc_}) by observing the viral state $v[k]$. A great advantage for the estimation of the spreading parameters $q, W$ is the linearity of the NIMFA equations (\ref{NIMFA_disc_}) with respect to $q, W$.
  
\item[Q2.] In view of the absence of a closed-form solution of the NIMFA difference equation (\ref{NIMFA_disc_}), \textit{is there an approximate and simpler description of the viral state evolution?} Of particular interest is a worst-case scenario of the viral spread, i.e. an upper bound of the viral state $v_i[k]$ for any node $i$ at any time $k$.

\item[Q3.] \textit{How quickly does the virus spread?} I.e. how fast does the viral state $v[k]$ approach the steady-state $v_\infty$?
\end{enumerate}
In this work, we give answers to the questions Q2 and Q3. In summary, the NIMFA system (\ref{NIMFA_disc_}) is a well-behaved and powerful model, which can be applied to a variety of epidemic phenomena due to the full heterogeneity of the spreading parameters $W, q$. In Section \ref{sec:related_work}, we review related work. The nomenclature and assumptions are introduced in Section \ref{sec:nomenclature} and Section \ref{sec:assumptions}, respectively. In Section \ref{subsec:alternativ_representation}, we analyse the viral state dynamics for large times $k$. We study the monotonicity of the viral state evolution in Section \ref{subsec:monotonicity}. In Section \ref{subsec:bounds_lti}, we derive upper and lower bounds on the viral state dynamics.

\section{Related Work} \label{sec:related_work}
The continuous-time NIMFA model (\ref{NIMFA_continuous}) with homogeneous spreading parameters was originally proposed in \cite{van2009virus}. An extension to more heterogeneous spreading parameters was provided in \cite{mieghem2014homogeneous}, with the constraint that there is a $\beta_i$ for every node $i$ such that for every node $j\neq i$ either $\beta_{ij} = \beta_{i}$ or $\beta_{ij}=0$.

The discrete-time NIMFA model with homogeneous spreading parameters has been studied in \cite{ahn2013global, pare2018analysis, prasse2018networkreconstruction}. The discrete-time NIMFA model (\ref{NIMFA_disc_}) with fully heterogeneous spreading parameters has been proposed by Par\'e \textit{et al.} \cite{pare2018analysis}, who showed that there is either one stable equilibrium, the healthy state $v[k]=0$, or there are two equilibria, the healthy state and a steady-state $v_\infty$ with positive components. Furthermore, the discrete-time NIMFA model (\ref{NIMFA_disc_stacked}) has been validated on data of real-world epidemics \cite{pare2018analysis}. We are not aware of results that assess the stability of the steady-state $v_\infty$ of the discrete-time NIMFA system (\ref{NIMFA_disc_}). 

\section{Nomenclature}\label{sec:nomenclature} 
 For a square matrix $M$, we denote the spectral radius by $\rho(M)$ and the eigenvalue with the largest real part by $\lambda_1(M)$. For an $N \times 1$ vector $z$ we define $z^l = (z^l_1, ..., z^l_N)^T$. For two $N \times 1$ vectors $y, z$, it holds $y > z$ or $y\ge z$ if $y_i > z_i$ or $y_i \ge z_i$, respectively, for every element $i=1, ..., N$.  The minimum of the discrete-time curing rates is denoted by $q_\textrm{min}= \textrm{min}\{q_1, ..., q_N\}$. We define the $N\times N$ matrix $R$ as
\begin{align}\label{definition_R}
R = I - \textup{\textrm{diag}}(q) + W.
\end{align}
The principal eigenvector of the matrix $R$ is denoted by $x_1$ and satisfies\footnote{Lemma \ref{lemma:R_irreducible} in Section \ref{subsec:alternativ_representation} states that there is indeed an eigenvalue $\lambda_1(R)$ of the matrix $R$ which equals the spectral radius $\rho(R)$.} 
\begin{align*}
R x_1 = \lambda_1(R)x_1 = \rho(R) x_1.
\end{align*}
Table \ref{table:nomenclature} summarises the nomenclature.

\begin{center}
  \begin{table*} \caption{Nomenclature \label{table:nomenclature}} 
  \centering
  \begin{tabular}{ | l | l | }
    \hline   
$\beta_{ij}$ & Continuous-time infection rate from group $j$ to group $i$ \\ \hline 
$\delta_i$ & Continuous-time curing rate of group $i$\\ \hline  
 $\textrm{diag}(x)$ & For a vector $x \in \mathbb{R}^N$, $\textrm{diag}(x)$ is the $N \times N$ diagonal matrix with $x$ on its diagonal \\ \hline   
$I$ & The $N \times N$ identity matrix \\ \hline       
  $\lambda_1(M)$& Eigenvalue with the largest real part of a square matrix $M$\\ \hline
 $N$& Number of nodes \\ \hline 
$q$ & Discrete-time curing rate vector, $q= (q_1, ..., q_N)^T$ and $q_i = \delta_i T$\\ \hline  
$q_\textrm{min}$ & Minimal discrete-time curing rate, $q_\textrm{min}= \textrm{min}\{q_1, ..., q_N\}$\\ \hline  
 $R$& $N\times N$ matrix $R = I - \textup{\textrm{diag}}(q) + W $\\ \hline 
$\rho(M)$ & Spectral radius of a square matrix $M$\\ \hline  
    $T$& Sampling time of the discrete-time NIMFA model\\ \hline
 $u$ & All-one vector $u = (1, ..., 1)^T \in \mathbb{R}^N$ \\ \hline     
$v[k]$ & Viral state $v[k] = (v_1[k], ..., v_N[k])^T$ at discrete time $k \in \mathbb{N}$, $v_i[k] \in [0, 1]$ for $i=1, ..., N$ \\ \hline
$v_\infty$ & Steady state vector, the non-zero equilibrium of (\ref{NIMFA_disc_}) \\ \hline
$\Delta v[k]$ & Difference of the viral state to the steady state, $\Delta v[k] = v[k] - v_\infty$ \\ \hline
$W$ & Discrete-time $N \times N$ infection rate matrix, $w_{ij} = \beta_{ij} T$ \\ \hline    
$x_1$ & Principal eigenvector of the matrix $R$, $R x_1 = \rho(R) x_1$ \\ \hline
  \end{tabular}
  \end{table*}
\end{center}

\section{Assumptions}\label{sec:assumptions} 
\begin{assumption} \label{assumption:spreading_parameters}
The curing rates are positive and the infection rates non-negative, i.e. $q_i >0$ and $w_{ij} \ge 0$ for all nodes $i, j$. 
\end{assumption}
\begin{assumption} \label{assumption:sampling_time}
 For every node $i=1, ..., N$, the sampling time $T>0$ satisfies 
\begin{align} \label{kjbnkaaass}
T \le T_\textup{\textrm{max}} = \frac{1}{\delta_i +  \sum^N_{j=1} \beta_{ij}}.
\end{align}  
\end{assumption}
The results of this work which rely on Assumption \ref{assumption:sampling_time} hold true if the sampling time is sufficiently small, which we consider a rather technical assumption. The particular choice of the bound (\ref{kjbnkaaass}) is due to Lemma \ref{lemma:nonnegativity} in Section \ref{subsec:alternativ_representation}. Furthermore, we make the following assumption on the initial viral state $v_i[1]$. 
\begin{assumption}\label{assumption:initial_state}
For every node $i=1, ..., N$, it holds that $0 \le v_i[1] \le v_{\infty, i}$.
\end{assumption}
At the beginning of the outbreak of an infectious disease, every group $i$ of individuals is almost disease-free. Thus, it is a reasonable assumption that the initial viral state $v_i[1]$ of every group $i$ is sufficiently small at the initial time $k=1$ such that that $v_i[1] \le v_{\infty, i}$ holds.
\begin{assumption} \label{assumption:connected_graph}
The infection rate matrix $W$ is irreducible.
\end{assumption}
Assumption \ref{assumption:connected_graph} holds if and only if the infection rate matrix $W$ corresponds to a strongly connected graph\footnote{In a strongly connected graph, there is a path from every node $i$ to any other node $j$.}. Finally, as shown in \cite{pare2018analysis}, Assumption \ref{assumption:above_threshold} avoids the trivial viral dynamics in which the virus dies out.
\begin{assumption} \label{assumption:above_threshold}
The spectral radius of the matrix $R$ is greater than one, i.e. $\rho \left( R \right) > 1$.
\end{assumption}

\section{Viral State Dynamics close to the Steady-State} \label{subsec:alternativ_representation}
For completeness, we recapitulate the results of Par\'e \textit{et al.} \cite{pare2018analysis} on the equilibria and the stability of the healthy state\footnote{
Theorem \ref{theorem:parestability} follows immediately from merging \cite[Theorems 1-2 and Proposition 2]{pare2018analysis}.}.
\begin{theorem}[\cite{pare2018analysis} ]\label{theorem:parestability}
Under the Assumptions \ref{assumption:spreading_parameters}, \ref{assumption:sampling_time} and \ref{assumption:connected_graph}, the following two statements hold true:
\begin{enumerate}
\item If  $\rho \left( R \right) \le 1$, then the healthy state $v[k]=0$ is the only equilibrium of the discrete-time NIMFA model (\ref{NIMFA_disc_stacked}). Furthermore, $v[k] \rightarrow 0$ when $k \rightarrow \infty$ for any initial viral state $v[1]$ with $0 \le v_i[1] \le 1$ for every node $i$. 

\item If $\rho \left( R \right) > 1$, then there are two equilibria of the discrete-time NIMFA model (\ref{NIMFA_disc_stacked}): The healthy state $v[k]=0$ and a steady-state $v_\infty$ with $v_{\infty, i}>0$ for every node $i$.
\end{enumerate}
\end{theorem}
We remark that the NIMFA model with \textit{homogeneous} spreading parameters \cite{van2009virus, van2011n} assumes that there is a scalar curing rate $\delta$ and a scalar infection rate $\beta$ such that $q_i = \delta$ and $\beta_{ij} = \beta a_{ij}$ for all nodes $i, j$, where $a_{ij}$ denote the elements of a symmetric and irreducible zero-one adjacency matrix $A$. For the NIMFA model with homogeneous spreading parameters, the condition $\rho \left( R \right) \le 1$ simplifies to $\tau \le \tau^{(1)}_c$ with the effective infection rate $\tau = \beta / \delta$ and the epidemic threshold $\tau^{(1)}_c = 1/\lambda_1(A)$. 
 
\begin{lemma}\label{lemma:R_irreducible}
Suppose that Assumptions \ref{assumption:spreading_parameters}, \ref{assumption:sampling_time} and \ref{assumption:connected_graph} hold. Then, the matrix $R$ is irreducible and non-negative. Hence, there is a real eigenvalue $\lambda_1(R)$ of the matrix $R$ which equals the spectral radius $\rho \left( R \right)$, and the principal eigenvector $x_1$ of the matrix $R$ is positive.
\end{lemma}
\begin{proof}
Appendix A. 
\end{proof}
We can generalise the bounds \cite{van2009virus,mieghem2014homogeneous} for the steady-state vector $v_\infty$ to the NIMFA model (\ref{NIMFA_disc_}) with heterogeneous spreading parameters.
\begin{lemma}\label{lemma:v_infty_bound}
Suppose that Assumptions \ref{assumption:spreading_parameters}, \ref{assumption:sampling_time}, \ref{assumption:connected_graph} and \ref{assumption:above_threshold} hold. Then, the steady-state $v_{\infty, i}$ of any node $i$ is bounded by
\begin{align*} 
1 - \frac{q_i}{\sum^N_{j=1} w_{ij}}  \le v_{\infty, i} \le 1 - \frac{q_i}{ q_i + \sum^N_{j=1} w_{ij} }.
\end{align*}
\end{lemma}
\begin{proof}
Appendix B.
\end{proof}

 We denote the difference of the viral state $v[k]$ to the steady state $v_\infty$ by $\Delta v[k] = v[k] - v_\infty$. By considering the difference $\Delta v[k]= v[k] - v_\infty$, we obtain an equivalent representation\footnote{Proposition \ref{proposition:delta_v_equation} is a generalisation of \cite[Proposition 3]{prasse2018networkreconstruction} to the NIMFA model with heterogeneous spreading parameters $q, W$.} of the discrete-time NIMFA equations (\ref{NIMFA_disc_}).
\begin{proposition}[NIMFA Equations as Difference to the Steady-State]\label{proposition:delta_v_equation}
Suppose that Assumptions \ref{assumption:spreading_parameters}, \ref{assumption:sampling_time}, \ref{assumption:connected_graph} and \ref{assumption:above_threshold} hold. Then, the difference $\Delta v[k]= v[k] - v_\infty$ from the viral state $v[k]$ to the steady state $v_\infty$ of the discrete-time NIMFA model (\ref{NIMFA_disc_stacked}) evolves according to
\begin{align} \label{delta_v_system}
\Delta v[k+1] & = F \Delta v[k]-  \textup{\textrm{diag}}(\Delta v[k] )  W  \Delta v[k],   
\end{align}
where the $N\times N$ matrix $F$ is given by
\begin{equation} \label{matrix_F}
F = I + \textup{\textrm{diag}}\left( \frac{q_1}{ v_{\infty, 1} - 1}, ...,  \frac{q_N}{ v_{\infty, N} - 1}\right) +  \textup{\textrm{diag}}(u- v_{\infty} ) W.
\end{equation}
\end{proposition}
\begin{proof}
Appendix C.
\end{proof} 
For a sufficiently small sampling time $T$, Lemma \ref{lemma:nonnegativity} states that every element of matrix $F$ is non-negative.
\begin{lemma} \label{lemma:nonnegativity}
Suppose that Assumptions \ref{assumption:spreading_parameters}, \ref{assumption:sampling_time}, \ref{assumption:connected_graph} and \ref{assumption:above_threshold} hold. Then, the $N \times N$ matrix $F$ defined by (\ref{matrix_F}) is non-negative, i.e. $(F)_{ij} \ge 0$ for every $i, j = 1, ..., N$.
\end{lemma}
\begin{proof}
Appendix D.
\end{proof}
Furthermore, Proposition \ref{proposition:delta_v_equation} leads to the following corollary\footnote{Corollary \ref{corollary:below_steady_state} is a generalisation of \cite[Corollary 1]{prasse2018networkreconstruction} to the NIMFA model with heterogeneous spreading parameters $q, W$.}.
\begin{corollary}\label{corollary:below_steady_state}
Suppose that Assumption \ref{assumption:spreading_parameters}--\ref{assumption:above_threshold} hold. Then, it holds that $v_i[k] \le v_{\infty, i}$ for every node $i$ at every time $k\ge 1$.
\end{corollary}
\begin{proof}
Appendix E.
\end{proof} 
In other words, Corollary \ref{corollary:below_steady_state} states that the set $\mathcal{V} = \{v | 0\le v_i \le v_{\infty, i}, ~\forall i = 1, ..., N\}$ is a \textit{positive invariant set} \cite{khalil1996nonlinear} of the NIMFA model (\ref{NIMFA_disc_}), i.e. if the initial viral state $v[1]$ is element of the set $\mathcal{V}$, then the viral state $v[k]$ will remain in the set $\mathcal{V}$ for $k \ge 1$. We emphasise that Corollary \ref{corollary:below_steady_state} does not imply that the viral state $v[k]$ increases monotonically.

To provide a graphical illustration of Corollary \ref{corollary:below_steady_state}, we generate a random network with $N=10$ nodes by creating a directed link $a_{ij} = 1$ from any node $j$ to any node $i$ with probability $0.25$, and we repeat this network generation if the resulting network is not strongly connected. If $a_{ij} = 1$, then we set the infection rate $w_{ij}$ to a uniformly distributed random number in $[0,1]$, and, if $a_{ij}=0$, then we set $w_{ij}=0$. The curing rate $q_i$ for every node $i$ is set to a uniformly distributed random number in $[0.95 c,1.05c]$, where $c = 10$ is a constant. If the spectral radius $\rho(R)\le 1+10^{-3}$, then we set the constant $c$ to $c/1.1$ and generate new curing rates $q$, and we repeat this generation of curing rates $q$ until $\rho(R)> 1+10^{-3}$. The sampling time $T$ is set to $T = T_\textup{\textrm{max}}/10$, given by (\ref{kjbnkaaass}). For every node $i$, the initial viral state $v_i[1]$ is set to a uniformly distributed random number in $[0,0.01v_{\infty, i}]$. Figure \ref{fig_not_monotonuous} depicts the resulting viral state traces $v_i[k]$ for every node $i$. As stated by Corollary \ref{corollary:below_steady_state}, the viral state $v[k]$ approaches the steady state $v_{\infty}$ from below without overshooting, but the viral state $v[k]$ is not strictly increasing. The absence of an overshoot is not evident, for instance in a Markovian SIS process an overshoot is possible \cite{van2016universality}.
\begin{figure}[h!]
	\centering
	 \includegraphics{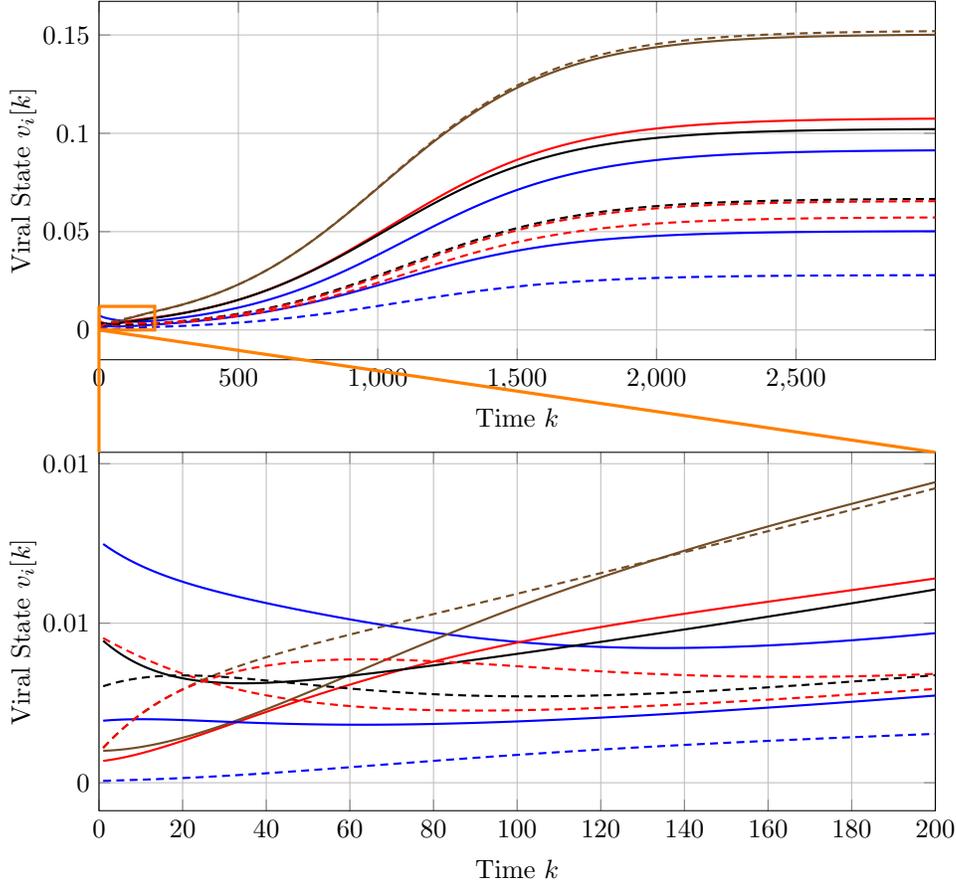}
	{\footnotesize\caption{The viral state traces $v_i[k]$, $i=1, ..., N$, for a directed network with $N=10$ nodes and heterogeneous spreading parameters $q, W$.
The lower sub-plot shows that the viral state $v_i[k]$ is not strictly increasing for every node $i$. \label{fig_not_monotonuous}}}
\end{figure}

For applications in which the initial viral state $v[1]$ is close to zero, the NIMFA equations (\ref{eq:LTI_sys_NIMFA_zero}) can be replaced by linear time-invariant (LTI) systems in two different regimes: On the one hand, it holds for small times $k$ that $v[k] \approx 0$. Hence, the representation (\ref{NIMFA_disc_stacked}) can be linearised around the origin $v[k] = 0$, which yields
\begin{align}\label{eq:LTI_sys_NIMFA_zero}
v [k + 1] \approx R v[k],
\end{align}
for small times $k$. On the other hand, if the viral state $v[k]$ is close to the steady-state $v_\infty$, which implies $\Delta v[k] \approx 0$, the representation (\ref{delta_v_system}) can be linearised around the origin $\Delta v[k] = 0$, which gives
\begin{align} \label{eq:LTI_sys_NIMFA_ss}
\Delta v[k+1] \approx F \Delta v[k].
\end{align}
Furthermore, we obtain that the steady-state $v_\infty$ is asymptotically stable\footnote{The steady-state $v_\infty$ is asymptotically stable if there exists an $\epsilon>0$ such that $\lVert v[1] - v_\infty \rVert <\epsilon$ implies that $v[k] \rightarrow v_\infty$ when $k\rightarrow \infty$.}. 
\begin{theorem}[Asymptotic Stability of the Steady-State] \label{theorem:stability}
Under the Assumptions \ref{assumption:spreading_parameters}, \ref{assumption:sampling_time}, \ref{assumption:connected_graph} and \ref{assumption:above_threshold}, the steady-state $v_\infty$ of the discrete-time NIMFA system (\ref{NIMFA_disc_stacked}) is asymptotically stable.
\end{theorem}
\begin{proof}
Appendix F.
\end{proof}

Ahn \textit{et al.} \cite{ahn2013global} gave a counterexample for which the steady-state $v_\infty$ of the discrete-time NIMFA system (\ref{NIMFA_disc_stacked}) is unstable. However, their counterexample does not satisfy Assumption \ref{assumption:sampling_time}. Hence, a sufficiently small sampling time $T$ is decisive for the stability of the discrete-time NIMFA model (\ref{NIMFA_disc_stacked}). (Par\'{e} \textit{et al.} \cite{pare2018analysis} observed that the counterexample in \cite{ahn2013global} violates the third assumption in \cite{pare2018analysis}, which is closely related to Assumption \ref{assumption:sampling_time}.)

\section{Monotonicity of the Viral State Dynamics}\label{subsec:monotonicity}

As stated by the property P6 in Section \ref{sec:introduction}, we will show that the viral state $v[k]$ is increasing, provided that the initial viral state $v[1]$ is small.
\begin{definition}[Strictly Increasing Viral State Evolution]
The viral state $v[k]$ is \emph{strictly increasing at time $k$} if $v[k+1] > v[k]$. The viral state $v[k]$ is \emph{globally strictly increasing} if $v[k]$ is strictly increasing at every time $k \ge 1$.
\end{definition}

Lemma \ref{lemma:monotonicity_induction} states an inductive property of the monotonicity.
\begin{lemma} \label{lemma:monotonicity_induction}
Under Assumptions \ref{assumption:spreading_parameters}--\ref{assumption:above_threshold}, the viral state $v[k]$ is strictly increasing at time $k$ if the viral state $v[k-1]$ is strictly increasing at time $k-1$. 
\end{lemma}
\begin{proof}
Appendix G.
\end{proof}
We obtain Theorem \ref{theorem:monotonicity}, which gives equivalent conditions to a globally strictly increasing viral state evolution.
\begin{theorem}[Monotonicity of the Viral State Evolution]\label{theorem:monotonicity}
Suppose that Assumptions \ref{assumption:spreading_parameters}--\ref{assumption:above_threshold} hold. Then, the viral state $v[k]$ is globally strictly increasing if and only if one of the following two (equivalent) statements holds.
\begin{enumerate}
\item The initial viral state $v[1]$ satisfies
\begin{align} \label{monotonicity_first_sta}
 \left( W - \textup{\textrm{diag}}(q)\right) v[1] >  \textup{\textrm{diag}}(q)\sum^\infty_{l = 2} v^l[1].
\end{align}

\item It holds
\begin{equation*}
\left(\textup{\textrm{diag}}\left( u - v_\infty \right)W\textup{\textrm{diag}}\left( u - v_\infty \right) - \textup{\textrm{diag}}(q) \right) z > \textup{\textrm{diag}}(q)\sum^\infty_{l = 2} z^l,
\end{equation*}
where the $N \times 1$ vector $z$ is given by
\begin{align*}
z_i = \frac{v_i[1] - v_{\infty, i}}{1 - v_{\infty, i}}, \quad i=1, ..., N.
\end{align*}
\end{enumerate}
\end{theorem}
\begin{proof}
Appendix H.
\end{proof}
From Theorem \ref{theorem:monotonicity}, we obtain a corollary which states sufficient conditions for a globally strictly increasing viral state.
\begin{corollary}\label{corollary_monotonicity_12}
Suppose Assumptions \ref{assumption:spreading_parameters}--\ref{assumption:above_threshold} hold and that the initial viral state $v[1]$ equals either
\begin{align}\label{v1_small_x1}
v[1] = \epsilon x_1 + \eta
\end{align}
or 
\begin{align*}
v[1] = (1 - \epsilon)v_\infty + \eta
\end{align*}
for some small $\epsilon >0$ and an $N \times 1$ vector $\eta$ whose norm $\lVert \eta \rVert_2 = \mathcal{O}(\epsilon^p)$ for some scalar $p>1$ which is independent of $\epsilon$. Then, there exists an $\epsilon >0$ such that the viral state $v[k]$ is globally strictly increasing.
\end{corollary}
\begin{proof}
Appendix I.
\end{proof}
Numerical simulations show that if the initial viral state $v[1]$ approaches zero from an \textit{arbitrary} direction, which differs from (\ref{v1_small_x1}), then the viral state $v[k]$ is in general \textit{not} globally strictly increasing. However, the simulations also indicate that, if the initial viral state $v[1]$ is small, then the viral state seems \quotes{almost} globally strictly increasing, which is illustrated by Figure \ref{fig_not_monotonuous} and motivates us to state Definition \ref{definition:quasi_increasing}. 
\begin{definition}[Quasi-Increasing Viral State Evolution]\label{definition:quasi_increasing}
Define $S_-$ as the set of times $k \ge 1$ at which the viral state $v[k]$ is not strictly increasing:
\begin{align*}
S_- = \left\{k\in \mathbb{N} ~| ~\exists i: ~ v_i[k+1]\le v_i[k] \right\}.
\end{align*}
Then, the viral state $v[k]$ is \emph{quasi-increasing with stringency $\epsilon$}, if the set $S_-$ is finite and 
\begin{align*}
\lVert v[k+1] - v[k] \rVert_2 \le \epsilon 
\end{align*}
for every time $k$ in $S_-$.
\end{definition}
Thus, a quasi-increasing viral state $v[k]$ is strictly increasing at every time $k$ not in the set $S_-$, and at the times $k$ in the finite set $S_-$, the viral state $v[k]$ is decreasing only within an $\epsilon$-stringency. For the viral state trace $v[k]$ depicted in Figure \ref{fig_not_monotonuous}, the set $S_-$ equals $S_-= \{1,2, ..., 165\}$. Theorem \ref{theorem:soft_monotonicity} states that the viral state $v[k]$ is quasi-increasing with an arbitrarily small stringency $\epsilon$, provided that the initial viral state $v[1]$ is sufficiently small.
\begin{theorem}\label{theorem:soft_monotonicity}
Suppose that Assumptions \ref{assumption:spreading_parameters}--\ref{assumption:above_threshold} hold and that $v[1]\neq 0$. Then, for any $\epsilon >0$ there is a $\vartheta(\epsilon)$ such that $\lVert v[1] \rVert_2 \le \vartheta(\epsilon)$ implies that the viral state $v[k]$ is quasi-increasing with stringency $\epsilon$. 
\end{theorem}
\begin{proof}
Appendix J.
\end{proof}

\section{Bounds on the Viral State Dynamics} \label{subsec:bounds_lti}
Due to the non-linearity of the NIMFA equations (\ref{NIMFA_disc_stacked}), an analysis of the exact viral state evolution is challenging. However, it is possible to upper and lower bound the viral state $v[k]$ by LTI systems, which allows for an \textit{approximate} analysis of the viral state evolution. As stated by Proposition \ref{proposition:upper_bound_lti_A}, the linearisation (\ref{eq:LTI_sys_NIMFA_zero}) of the NIMFA model around zero directly yields an upper bound on the viral state $v[k]$. 
\begin{proposition}[First Upper Bound] \label{proposition:upper_bound_lti_A}
Suppose that Assumptions \ref{assumption:spreading_parameters}--\ref{assumption:initial_state} hold and define the LTI system 
\begin{align} \label{lti_upper_bound_A}
v^{(1)}_\textup{\textrm{ub}} [k + 1] =R v^{(1)}_\textup{\textrm{ub}}[k] \quad k \ge 1,
\end{align}
where the matrix $R$ is given by (\ref{definition_R}). If $v^{(1)}_{\textup{\textrm{ub}}}[1] \ge v[1]$, then it holds that $v^{(1)}_{\textup{\textrm{ub}}}[k] \ge v[k]$ at every time $k \ge 1$.  If $\rho \left( R \right) \ge 1$, then the LTI system (\ref{lti_upper_bound_A}) is unstable. If $\rho \left( R\right) < 1$, then the LTI system (\ref{lti_upper_bound_A}) is asymptotically stable.
\end{proposition}
\begin{proof}
Appendix K.
\end{proof}
 Additionally to the upper bound in Proposition \ref{proposition:upper_bound_lti_A}, the linearisation (\ref{eq:LTI_sys_NIMFA_ss}) of the NIMFA model around the steady-state $v_\infty$ yields another upper bound on the viral state $v[k]$, as stated by Proposition \ref{proposition:upper_bound_lti_B}.
\begin{proposition}[Second Upper Bound] \label{proposition:upper_bound_lti_B}
Under the Assumptions \ref{assumption:spreading_parameters}--\ref{assumption:above_threshold}, define the LTI system 
\begin{align} \label{lti_upper_bound_B}
\Delta v_\textup{\textrm{ub}}[k+1] & = F \Delta v_\textup{\textrm{ub}}[k] \quad k \ge 1,
\end{align}
where the $N \times N$ matrix $F$ is given by (\ref{matrix_F}). Then, the following statements hold true:
\begin{enumerate}
\item  If $\Delta v_{\textup{\textrm{ub}}}[1] \ge \Delta v[1]$, then it holds that $\Delta v_\textup{\textrm{ub}}[k] \ge \Delta v[k]$ at every time $k \ge 1$.

\item If $\Delta v_\textup{\textrm{ub}}[1] \le 0$, then it holds that $\Delta v_\textup{\textrm{ub}}[k] \le 0$ at every time $k$.
\end{enumerate}
\end{proposition}
\begin{proof}
Appendix L.
\end{proof}
 Hence, the LTI system (\ref{lti_upper_bound_B}) yields the upper bound 
 \begin{align*}
 v^{(2)}_\textrm{ub}[k] = \Delta v_\textup{\textrm{ub}}[k] + v_{\infty} \ge v[k] 
 \end{align*}
 on the viral state $v[k]$ at every time $k$. If Assumption \ref{assumption:initial_state} holds and $0 \ge \Delta v_{\textup{\textrm{ub}}}[1] = \Delta v[1]$, then it holds that $0 \ge \Delta v_{\textup{\textrm{ub}}}[k] \ge \Delta v[k]$ for every time $k$. Thus, the convergence of $\Delta v[k]$ to 0 implies the convergence of $\Delta v_{\textup{\textrm{ub}}}[k]$ to 0. The upper bound of Proposition \ref{proposition:upper_bound_lti_A} is tight when the viral state $v[k]$ is small, and the upper bound of Proposition \ref{proposition:upper_bound_lti_B} is tight when the viral state $v[k]$ is close to the steady-state $v_\infty$. We combine Propositions \ref{proposition:upper_bound_lti_A} and \ref{proposition:upper_bound_lti_B} to obtain a tighter upper bound, for every node $i=1, ..., N$, as
\begin{align}\label{v_ub_1_and_2}
v_{\textrm{ub}, i}[k] = \textrm{min}\{v^{(1)}_{\textrm{ub}, i}[k], v^{(2)}_{\textrm{ub}, i}[k]\}.
\end{align} 
 Finally, Proposition \ref{proposition:lower_bound_lti} provides a lower bound on the viral state $v[k]$. 
\begin{proposition}[Lower Bound] \label{proposition:lower_bound_lti}
Suppose that Assumptions \ref{assumption:spreading_parameters}--\ref{assumption:above_threshold} hold and let there be an $N\times 1$ vector $v_\textup{\textrm{min}}>0$ such that $v[k] \ge v_\textup{\textrm{min}}$ holds at every time $k\ge 1$. Furthermore, let $\Delta v_{\textup{\textrm{lb}}}[1] = \Delta v[1]$ and define the LTI system 
\begin{align} \label{lti_lower_bound}
\Delta v_\textup{\textrm{lb}}[k+1] & = F_\textup{\textrm{lb}} \Delta v_\textup{\textrm{lb}}[k] \quad k \ge 1,
\end{align}
where the $N\times N$ matrix $F_\textup{\textrm{lb}}$ is given by
\begin{equation*} 
F_\textup{\textrm{lb}} = I + \textup{\textrm{diag}}\left( \frac{q_1}{ v_{\infty, 1} - 1}, ...,  \frac{q_N}{ v_{\infty, N} - 1}\right) +  \textup{\textrm{diag}}\left(u- v_\textup{\textrm{min}} \right) W.
\end{equation*}
Then, the following statements hold true:
\begin{enumerate}
\item It holds that $\Delta v_{\textup{\textrm{lb}}}[k] \le \Delta v [k] \le 0$ at every time $k \ge 1$. 

\item Denote $\gamma = \textup{\textrm{min}}\{v_{\textup{\textrm{min}}, 1}, ..., v_{\textup{\textrm{min}}, N}\}$. Then, it holds
\begin{align*}
\Delta v_\textup{\textrm{lb}}[k] \ge - \left( 1 - q_\textup{\textrm{min}}\frac{\gamma}{1 - \gamma}\right)^{k-1} v_\infty.
\end{align*}
Hence, $\Delta v_{\textup{\textrm{lb}}}[k] \rightarrow 0$ when $k \rightarrow \infty$.
\end{enumerate}
\end{proposition}
\begin{proof}
Appendix M.
\end{proof}
Hence, the LTI system (\ref{lti_lower_bound}) yields the lower bound 
\begin{align} \label{v_lb}
v_\textrm{lb}[k]  \le \Delta v_\textrm{lb}[k] + v_{\infty}
\end{align}
on the viral state $v[k]$ at every time $k$. In particular, if the viral state $v[k]$ is strictly increasing at every time $k \ge 1$, as discussed in Section \ref{subsec:monotonicity}, then the vector $v_\textrm{min}$ can be chosen as $v_\textrm{min} = v[1]$. Proposition \ref{proposition:lower_bound_lti} implies that the viral state $v[k]$ converges to the steady-state $v_\infty$ exponentially fast:
\begin{corollary}[Steady-State is Globally Exponentially Stable] \label{corollary:exponential_stability}
Suppose that Assumptions \ref{assumption:spreading_parameters}--\ref{assumption:above_threshold} hold. Then, for any initial viral state $v[1] > 0$ there is a constant $\alpha < 1$ such that 
\begin{align} \label{jkhbjkhass}
\lVert v[k] - v_\infty \rVert_2 \le \lVert v_\infty \rVert_2 \alpha^{k-1} \quad \forall k\ge 1.
\end{align}
If the viral state $v[k]$ is furthermore globally strictly increasing (cf. Theorem \ref{theorem:monotonicity}), then (\ref{jkhbjkhass}) is satisfied for $\alpha =  (1  - q_\textup{\textrm{min}} \gamma / (1- \gamma))$, where $\gamma = \textup{\textrm{min}}\left\{v_1[1], ..., v_N[1]\right\}$.
\end{corollary}
\begin{proof}
Appendix N.
\end{proof}

In the susceptible-infected-susceptible (SIS) epidemic process\cite{daley2001epidemic, pastor2015epidemic}, the \textit{hitting time} $T_{H_n}$ is the first time when the SIS process reaches a state with $n$ infected nodes. As argued in \cite{he2018spreading}, the average hitting time $\textrm{E}[T_{H_n}]$ scales exponentially with respect to the number $n$ of infected nodes, which is in agreement with the exponential convergence to the steady state $v_\infty$ for the NIMFA epidemic model\footnote{For an SIS process, the \textit{spreading time} \cite{van2014time} is another measure for the time of convergence to the metastable state. For the spreading time, the convergence to the metastable state is defined differently for every realisation of the same SIS epidemic process. Hence, the spreading time is subject to random fluctuations, which approximately follow a lognormal distribution \cite{he2018spreading}, contrary to the deterministic NIMFA model (\ref{NIMFA_disc_}) and the average hitting time $\textrm{E}[T_{H_n}]$ of an SIS process.}.

We provide a numerical evaluation of the upper bound $v_\textrm{ub}[k]$, given by (\ref{v_ub_1_and_2}), and the lower bound $v_\textrm{lb}[k]$, given by (\ref{v_lb}). We generate a directed Erd{\H o}s-R{\'e}nyi random graph with $N=500$ nodes by creating a directed link $a_{ij} = 1$ from any node $j$ to any node $i$ with link probability $0.05$. We generate another graph if the resulting graph is not strongly connected. If $a_{ij} = 1$, then we set the infection rate $w_{ij}$ to a uniformly distributed random number in $[0,1]$, and, if $a_{ij}=0$, then we set $w_{ij}=0$. The curing rate $q_i$ for every node $i$ is set to a uniformly distributed random number in $[0.95 c,1.05c]$, where $c = 10$ is a constant. If the spectral radius $\rho(R)\le 1+10^{-5}$, then we set the constant $c$ to $c/1.005$ and generate new curing rates $q$, and we repeat this generation of curing rates $q$ until $\rho(R)> 1+10^{-5}$. The sampling time $T$ is set to $T = T_\textup{\textrm{max}}/20$, given by (\ref{kjbnkaaass}). For every node $i$, the initial viral state $v_i[1]$ is set to a uniformly distributed random number $[0,0.1v_{\infty, i}]$. Figure \ref{fig_V_bounds} illustrates the goodness of the bounds $v_\textrm{ub}[k]$ and $v_\textrm{lb}[k]$ on the viral state $v[k]$, when the bounds are initialised at different times $k_0\ge 1$, i.e. $v_\textrm{lb}[k_0] = v[k_0] = v_\textrm{ub}[k_0]$. For a small initialisation time $k_0$, the upper bound $v_\textrm{ub}[k]$ results in a reasonable fit, whereas the lower bound $v_\textrm{lb}[k]$ does not perform well. If the initialisation time $k_0$ is greater, then both bounds $v_\textrm{lb}[k]$, $v_\textrm{ub}[k]$ give a tight fit to the exact viral state $v[k]$.

\section{Conclusions}
In this work, we analysed the discrete-time NIMFA epidemic model with heterogeneous spreading parameters on directed graphs. Our contribution is threefold.

First, we gave an alternative and equivalent representation of the NIMFA equations. We proved that the steady-state $v_\infty$ is asymptotically stable. Furthermore, we showed that the viral state $v[k]$ approaches the steady-state $v_\infty$ without overshooting, which is a phenomenon that occurs in many real-world epidemics.

Second, provided that the initial viral state $v[1]$ is sufficiently small, we showed that the viral state $v[k]$ is increasing, which, again, is an important characteristic of real-world epidemics.

Third, we derived linear systems that give upper and lower bounds on the viral state $v[k]$, and we proved that the viral state $v[k]$ converges to the steady-state $v_\infty$ exponentially fast.

In conclusion, we have shown that the discrete-time NIMFA epidemic model captures the qualitative behaviour of real-world epidemics. Due to the heterogeneity of the spreading parameters and the directedness of the underlying contact network, the NIMFA system allows for modelling a broad spectrum of real-world spreading phenomena.

\section*{Acknowledgements}\label{sec:acknowledgements}                                                                          
We are grateful to Qiang Liu for helpful discussions on this material.

\begin{figure}[H]
	\centering
	 \includegraphics{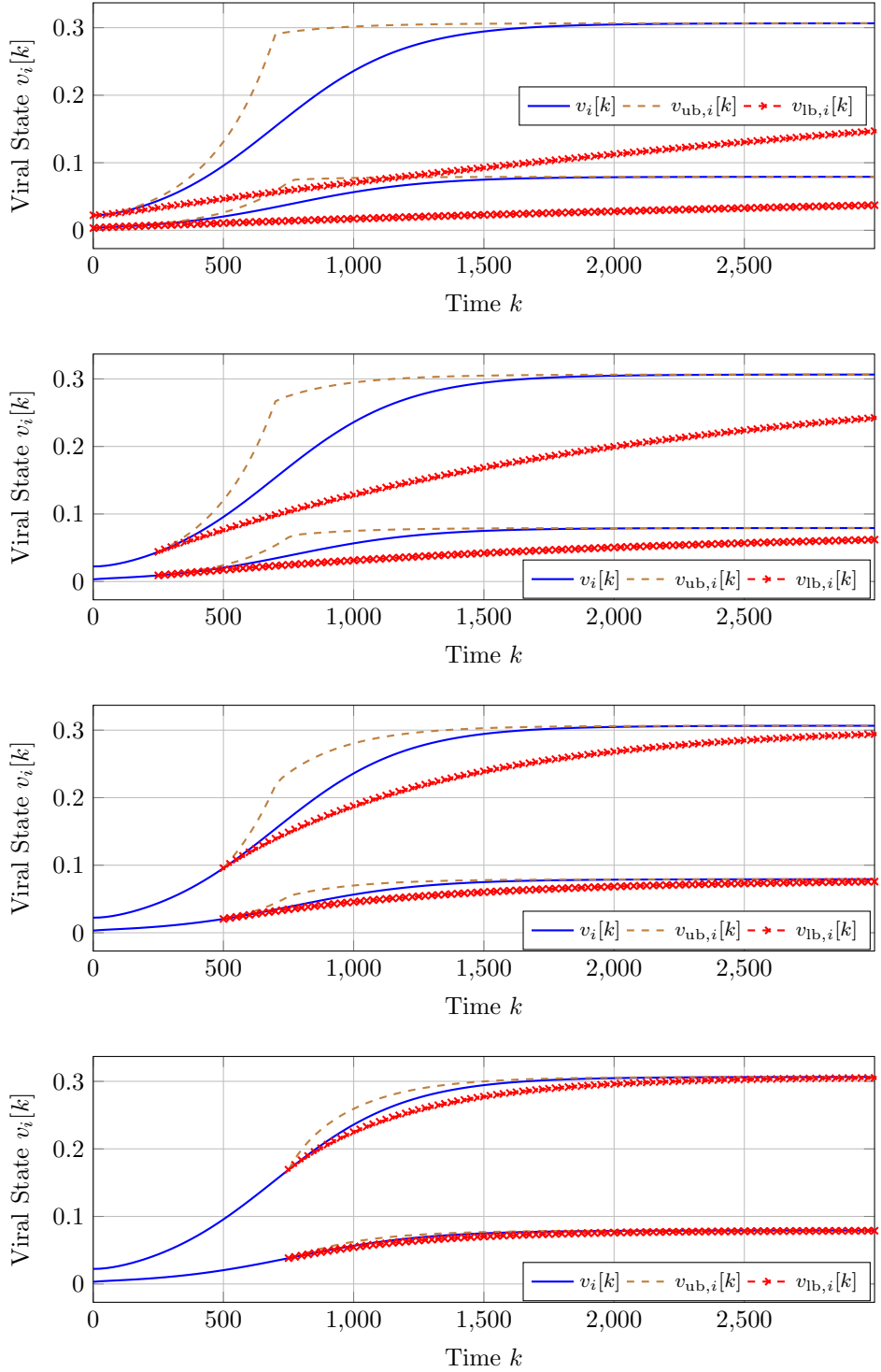}
	{\footnotesize\caption{For a directed Erd{\H o}s-R{\'e}nyi random graph with $N=500$ nodes and heterogeneous spreading parameters $q, W$, the fit of the lower bound $v_\textrm{lb}[k]$ and the upper bound $v_\textrm{ub}[k]$ on the viral state $v[k]$ is depicted. Each of the four sub-plots shows two viral state traces $v_i[k]$ and the corresponding bounds of the two nodes with the maximal and minimal steady-state $v_{\infty, i}$, respectively. From top to bottom, the sub-plots correspond to an initialisation of the bounds $v_\textrm{lb}[k_0] = v[k_0] = v_\textrm{ub}[k_0]$ at time $k_0 = 1$, $k_0 = 250$, $k_0 = 500$ and $k_0 = 750$, respectively.} \label{fig_V_bounds}}
\end{figure}

\appendix

\section{Proof of Lemma 1}
The elements of the matrix $R$, defined by (\ref{definition_R}), equal
\begin{align*}
R_{ij} = \begin{cases}
 1 - q_i +w_{ii}\quad &\text{if} \quad i = j, \\
 w_{ij}&\text{if} \quad i \neq j. 
 \end{cases}
\end{align*}
Under Assumption \ref{assumption:spreading_parameters}, it holds $w_{ij} \ge 0$ for all nodes $i, j$. Thus, the off-diagonal entries of the matrix $R$ are non-negative. For the diagonal entries of the matrix $R$, it holds
\begin{align*}
R_{ii} = 1 - q_i +w_{ii} \ge 1 - \delta_i T,
\end{align*}
since $w_{ii}\ge 0$ and $q_i = \delta_i T$. From Assumption \ref{assumption:sampling_time}, we further obtain that
\begin{align*}
R_{ii} \ge 1 - \delta_i  \frac{1}{\delta_i + \sum^N_{j=1}\beta_{ij}} \ge 0.
\end{align*}
Hence, the matrix $R$ is non-negative. Furthermore, the matrix $R$ is irreducible, which follows from the irreducibility of the infection rate matrix $W$ under Assumption \ref{assumption:connected_graph}. From the Perron-Frobenius Theorem \cite{horn1990matrix} follows that there is a real eigenvalue $\lambda_1(R)$ of the matrix $R$ which equals the spectral radius $\rho(R)$ and that the principal eigenvector $x_1$ is positive.
 
\section{Proof of Lemma 2}
The proof is analogous to the proof in \cite[Theorem 5 and Lemma 9]{van2009virus}. From the steady-state equation (\ref{steady_state_componentwise}), we obtain that 
\begin{align*} 
v_{\infty, i}\left( q_i + \sum^N_{j=1} w_{ij} v_{\infty, j}  \right) = \sum^N_{j=1} w_{ij} v_{\infty, j}. 
\end{align*}
Hence, it holds that
\begin{align*} 
v_{\infty, i}= \frac{\sum^N_{j=1} w_{ij} v_{\infty, j}}{ q_i + \sum^N_{j=1} w_{ij} v_{\infty, j}  }, 
\end{align*}
which equals
\begin{align} \label{kbhjkhbssss}
v_{\infty, i}= 1 - \frac{q_i}{ q_i + \sum^N_{j=1} w_{ij} v_{\infty, j}  } .
\end{align}
Since $v_{\infty, j} \le 1$ for every node $j$, we obtain an upper bound on the steady-state $v_{\infty, i}$ of node $i$ as
\begin{align*} 
v_{\infty, i} \le 1 - \frac{q_i}{ q_i + \sum^N_{j=1} w_{ij} } .
\end{align*}
We denote the minimum of the steady-state vector by
\begin{align*}
v_{\infty, \textrm{min}} = \textrm{min}\{v_{\infty, 1}, ..., v_{\infty, N}\}.
\end{align*}
Theorem \ref{theorem:parestability} implies that $v_{\infty, \textrm{min}} >0$. Assuming that the minimum $v_{\infty, \textrm{min}}$ occurs at node $i$, we obtain from (\ref{kbhjkhbssss}) that 
\begin{align*} 
v_{\infty, \textrm{min}}&= 1 - \frac{q_i}{ q_i + \sum^N_{j=1} w_{ij} v_{\infty, j}  }  \ge 1 - \frac{q_i}{ q_i + v_{\infty, \textrm{min}} \sum^N_{j=1} w_{ij} }.
\end{align*}
Hence, it holds
\begin{align*} 
v_{\infty, \textrm{min}}\ge \frac{v_{\infty, \textrm{min}} \sum^N_{j=1} w_{ij}}{ q_i + v_{\infty, \textrm{min}} \sum^N_{j=1} w_{ij} },
\end{align*}
from which we obtain that 
\begin{align*} 
 v_{\infty, \textrm{min}}  \ge 1 - \frac{q_i}{\sum^N_{j=1} w_{ij}}.
\end{align*}

\section{Proof of Proposition 1}
Since $\Delta v_i[k+1]= v_i[k+1] -v_{\infty, i}$, the evolution of the difference $\Delta v_i[k]$ over time $k$ can be stated with the NIMFA equations (2) as
\begin{equation}\label{delta_vi_ss}
\Delta v_i[k+1] = (1 - q_i) v_i[k] + \sum^N_{j=1} w_{ij}v_j[k] - v_i[k] \sum^N_{j=1} w_{ij}v_j[k]  - v_{\infty, i}.
\end{equation}
We would like to express the difference $\Delta v_i[k+1]$ at the next time $k+1$ only in dependency of the difference $\Delta v[k]$ at the current time $k$ and the constant steady state $v_\infty$. The steady state $v_\infty$ is given by (\ref{steady_state_componentwise}) and satisfies
\begin{align} \label{steadyState_f}
v_{\infty ,i} = (1 - q_i)v_{\infty, i} + (1 - v_{\infty,i}) \sum^N_{j=1} w_{ij} v_{\infty, j},
\end{align}
for all nodes $i$. We insert (\ref{steadyState_f}) in (\ref{delta_vi_ss}) and obtain
\begin{align} 
\Delta v_i[k+1] = &(1 - q_i) v_i[k] + \sum^N_{j=1} w_{ij}v_j[k] - v_i[k] \sum^N_{j=1} w_{ij}v_j[k] \nonumber \\
 &-(1 - q_i)v_{\infty, i} -  \sum^N_{j=1} w_{ij} v_{\infty, j} + v_{\infty, i}\sum^N_{j=1} w_{ij} v_{\infty, j}. \label{oijoijoijo}
  \end{align}
Since $\Delta v_i[k] = v_i[k]- v_{\infty,i}$, we can express (\ref{oijoijoijo}) more compactly as
\begin{equation} \label{eqljknlnklsdgf}
\Delta v_i[k+1] = (1 - q_i)\Delta v_i[k]+ \sum^N_{j=1} w_{i j} \Delta v_j[k] -  \sum^N_{j=1} w_{i j} \left( v_i[k]v_j[k] - v_{\infty, i} v_{\infty, j} \right).
\end{equation}
The first two addends in (\ref{eqljknlnklsdgf}) are already in the desired form: they depend on the difference $\Delta v[k]$ but not on the viral state $v[k]$ at time $k$. To replace the viral state $v[k]$ in the last term of (\ref{eqljknlnklsdgf}) by an expression of the difference $\Delta v[k]$, we observe that
\begin{equation} \label{sljknkljnd}
 v_i[k]v_j[k] - v_{\infty, i} v_{\infty, j}=  \Delta v_i[k] \Delta v_j[k] +  \Delta v_i[k] v_{\infty, j}   + v_{\infty, i}  \Delta v_j[k], 
\end{equation}
since $v_i[k] = \Delta v_i[k]+ v_{\infty,i}$. Inserting (\ref{sljknkljnd}) in (\ref{eqljknlnklsdgf}) yields 
\begin{align} 
\Delta v_i[k+1] = &\left( 1 - q_i  - \sum^N_{j=1} w_{i j}  v_{\infty, j} \right)\Delta v_i[k] \nonumber\\
&+ (1- v_{\infty, i} )  \sum^N_{j=1} w_{i j} \Delta v_j[k]  - \Delta v_i[k] \sum^N_{j=1}w_{i j}  \Delta v_j[k].   \label{ljnljjjjsssafff}
\end{align}
The expression (\ref{ljnljjjjsssafff}) can be further simplified. The steady-state equation (\ref{steady_state_componentwise}) is equivalent to  
\begin{align} \label{stads}
\sum^N_{j=1} w_{ij} v_{\infty, j} = q_i \left( \frac{1}{1 - v_{\infty, i}} -1\right).
\end{align}
From (\ref{stads}) follows that (\ref{ljnljjjjsssafff}) is equivalent to 
\begin{equation} \label{delta_vi_single}
\Delta v_i[k+1] = \left( 1  + \frac{q_i}{ v_{\infty, i} - 1}\right) \Delta v_i[k]+ (1- v_{\infty, i} )  \sum^N_{j=1} w_{i j} \Delta v_j[k]  -  \Delta v_i[k] \sum^N_{j=1}w_{i j}  \Delta v_j[k] .
\end{equation}
Stacking equation (\ref{delta_vi_single}) for all nodes $i = 1, ..., N$ completes the proof.

\section{Proof of Lemma 3}
We consider the elements of the matrix $F$. For $i \neq j$ it holds that
\begin{align*}
F_{ij} = \left(1 - v_{\infty, i}\right) w_{ij} \ge 0,
\end{align*}
since $ v_{\infty, i} \le 1$ and $w_{ij} \ge 0$. The diagonal elements of the matrix $F$ equal
\begin{align*}
F_{ii}= 1 + \frac{q_i}{v_{\infty, i} -1} + \left(1 - v_{\infty, i}\right) w_{ii}, \quad i=1, ..., N.
\end{align*} 
Since $w_{ii} \ge 0$, we obtain that
\begin{align}\label{khbikasss}
F_{ii} \ge 1 + \frac{q_i}{v_{\infty, i} -1}.
\end{align}
We proceed the proof of Lemma \ref{lemma:nonnegativity} by showing that the right hand side of (\ref{khbikasss}) is non-negative, i.e. by showing that
\begin{align*}
1 + \frac{q_i}{v_{\infty, i} -1} \ge 0,
\end{align*}
which is equivalent to 
\begin{align} \label{eq:erasff}
v_{\infty, i} \le 1 - q_i.
\end{align}
By using the upper bound on the viral state $v_{\infty, i}$ provided by Lemma \ref{lemma:v_infty_bound}, we obtain that a sufficient condition for (\ref{eq:erasff}) is
\begin{align*} 
1 - \frac{q_i}{ q_i + \sum^N_{j=1} w_{ij} } \le 1 - q_i,
\end{align*}
which is equivalent to 
\begin{align*} 
 q_i + \sum^N_{j=1} w_{ij}  \le  1.
 \end{align*}
From $q_i = \delta_i T$ and $w_{ij} = \beta_{ij} T$, we finally obtain that 
\begin{align*}
T \le \frac{1}{\delta_i +  \sum^N_{j=1} \beta_{ij}},
\end{align*} 
which holds true under Assumption \ref{assumption:sampling_time}, is a sufficient condition for $F_{ii} \ge 0$.

\section{Proof of Corollary 1}
We rewrite equation (\ref{ljnljjjjsssafff}) to obtain
\begin{align} \label{lkjnkja}
\Delta v_i[k+1] & = g_i[k] + h_i[k] \Delta v_i[k], 
\end{align}
where the terms $g_i[k]$ and $h_i[k]$ are given by 
\begin{align}\label{definition_gi}
g_i[k] = (1- v_{\infty, i} )  \sum^N_{j=1} w_{i j} \Delta v_j[k]    
\end{align}
and
\begin{align*}
h_i[k] =  1 - q_i  - \sum^N_{j=1}w_{i j}  \left( v_{\infty, j} +\Delta v_j[k]\right)
\end{align*}
for every node $i$. Since $w_{i j} \ge 0$ and $(1- v_{\infty, i} ) \ge 0$, the definition (\ref{definition_gi}) of $g_i[k]$ shows that
\begin{align*}
\Delta v_j[k] \le 0 ~ \forall j \Rightarrow g_i[k]\le 0.
\end{align*}
Furthermore, by the definition of $\Delta v_j[k]=  v_j[k]- v_{\infty, j}$ and since $ v_j[k] \in [0,1]$, it holds that
\begin{align} \label{klkjnlljnklnnlkn}
h_i[k] = 1 - q_i  -  \sum^N_{j=1}w_{i j}  v_j[k] \ge 1 - q_i -\sum^N_{j=1}w_{i j}.
\end{align}
Assumption \ref{assumption:sampling_time} states that $q_i + \sum^N_{j=1}w_{i j}\le 1$, and, hence, (\ref{klkjnlljnklnnlkn}) implies that $h_i[k] \ge 0$. From $h_i[k] \ge 0$ and $g_i[k] \le 0$ if $\Delta v_i[k] \le 0$ for all nodes $i$ and (\ref{lkjnkja}) follows that: $\Delta v_i[k] \le 0$ for all nodes $i$ implies $\Delta v_i[k+1] \le 0$ for all nodes $i$. Hence, we obtain by induction that $\Delta v_i[1] \le 0$ for all nodes $i$ implies $\Delta v_i[k] \le 0$ for all nodes $i$ at every time $k\ge 1$, which proves Corollary 1. (Analogously, we can prove that $\Delta v_i[0] \ge 0$ for all nodes $i$ implies $\Delta v_i[k] \ge 0$ for all nodes $i$ at every time $k\ge 1$.)

\section{Proof of Theorem 2}
The discrete-time NIMFA system (\ref{NIMFA_disc_stacked}) is asymptotically stable at the steady-state $v_\infty$ if the linearisation (\ref{eq:LTI_sys_NIMFA_ss}) at $\Delta v[k] = 0$ is stable \cite{khalil1996nonlinear}. The LTI system (\ref{eq:LTI_sys_NIMFA_ss}) is stable if the magnitudes of all the eigenvalues of its $N\times N$ system matrix $F$ are smaller than one, which is equivalent to $\rho(F)<1$ by the definition of the spectral radius $\rho(F)$. Lemma \ref{lemma:nonnegativity} states that the matrix $F$ is non-negative. Hence, the spectral radius $\rho(F)$ is upper bounded by \cite[Theorem 8.1.26.]{horn1990matrix}
\begin{align}\label{kjnbkjas}
\rho(F) \le \underset{i=1, ..., N}{\textrm{max}} \quad \frac{1}{y_i} \sum^N_{j=1} F_{ij} y_j
\end{align}
for any $N\times 1$ vector $y >0$. It holds $v_\infty> 0$, and by setting $y=v_\infty$, we obtain from (\ref{kjnbkjas}) that
\begin{align}\label{khjbkss}
\rho(F) \le \underset{i=1, ..., N}{\textrm{max}} \quad \frac{1}{v_{\infty, i}} \sum^N_{j=1} F_{ij} v_{\infty, j}.
\end{align}
From the definition (\ref{matrix_F}) of the matrix $F$ follows that
\begin{align*}
 \sum^N_{j=1} F_{ij} v_{\infty, j} &= v_{\infty, i} -  q_i \frac{ v_{\infty, i}}{ 1 -v_{\infty, i}}  +  \left(1- v_{\infty, i} \right) \sum^N_{j=1} w_{ij} v_{\infty, j} \\
 &= v_{\infty, i} -  q_i \frac{ v_{\infty, i}}{ 1 -v_{\infty, i}}  +   q_i v_{\infty, i},
\end{align*}
where the last equality follows from the steady-state equation (\ref{steady_state_componentwise}). Thus, the upper bound (\ref{khjbkss}) on the spectral radius $\rho(F)$ becomes
\begin{align*}
\rho(F) \le \underset{i=1, ..., N}{\textrm{max}} \quad 1 - q_i \frac{ v_{\infty, i} }{ 1 -v_{\infty, i}} < 1,
\end{align*}
since $q_i>0$ and $v_{\infty, i}>0$ for every node $i$. From $\rho(F)<1$ follows that the linearisation (\ref{eq:LTI_sys_NIMFA_ss}) is stable, and, hence, the discrete-time NIMFA system (\ref{NIMFA_disc_}) is asymptotically stable at the steady-state $v_\infty$.

\section{Proof of Lemma 4}
Since $\Delta v[k] = v[k] - v_\infty$, the viral state $v[k]$ is strictly increasing at time $k$ if and only if the difference $\Delta v[k]$ is strictly increasing at time $k$. Thus, it holds that 
\begin{align}\label{ljnkjndddd}
\Delta v_j[k] > \Delta v_j[k-1], \quad j= 1,..., N,
\end{align}
since the viral state $v[k-1]$ is assumed to be strictly increasing at time $k-1$. From Proposition \ref{proposition:delta_v_equation} follows that
\begin{align} \label{kljnkjbnkaasss}
\Delta v_i[k+1] - \Delta v_i[k] = &\sum^N_{j = 1} F_{ij} \left(\Delta v_j[k] - \Delta v_j[k-1] \right) \nonumber \\
 &+ \sum^N_{j=1} w_{ij}\left(\Delta v_i[k-1]\Delta v_j[k-1] - \Delta v_i[k] \Delta v_j[k]\right). 
\end{align}
As stated by Lemma \ref{lemma:nonnegativity}, the matrix $F$ is non-negative under Assumption \ref{assumption:sampling_time}. Thus, we obtain from $F_{ij}\ge 0$ and (\ref{ljnkjndddd}) that the first sum in (\ref{kljnkjbnkaasss}) is positive. Regarding the second sum in (\ref{kljnkjbnkaasss}), we observe that 
\begin{equation} \label{kljnljknlkasss}
\Delta v_i[k-1]\Delta v_j[k-1] - \Delta v_i[k] \Delta v_j[k]> \Delta v_i[k-1]\Delta v_j[k-1] - \Delta v_i[k] \Delta v_j[k-1] 
\end{equation}
due to (\ref{ljnkjndddd}) and since $\Delta v_i[k]\le 0$ holds for every node $i$ under Assumption \ref{assumption:initial_state} as stated by Corollary \ref{corollary:below_steady_state}. With (\ref{ljnkjndddd}) and $\Delta v_i[k]\le 0$ for every node $i$, we obtain from (\ref{kljnljknlkasss}) that
\begin{align*} 
\Delta v_i[k-1]\Delta v_j[k-1] - \Delta v_i[k] \Delta v_j[k] >0.
\end{align*}
Hence, since $w_{ij}\ge 0$ for every nodes $i, j$, both sums in (\ref{kljnkjbnkaasss}) are positive, which implies that $\Delta v_i[k+1] > \Delta v_i[k]$ for every node $i$.

\section{Proof of Theorem 3}
Lemma \ref{lemma:monotonicity_induction} states that $v[k+1] > v[k]$ implies $v[k+2] > v[k+1]$ for any time $k$. Thus, by showing that $v[2]>v[1]$, we show (inductively) that $v[k+1] > v[k]$ for every time $k$. We prove the two statements of Theorem \ref{theorem:monotonicity} in Subsection \ref{subsec:statement1} and Subsection \ref{subsec:statement2}, respectively.
\subsection{Statement 1}\label{subsec:statement1}
From the NIMFA equations (\ref{NIMFA_disc_}) follows that
\begin{align*}
v_i[2] - v_i[1] = - q_i v_i[1] + (1 - v_i[1])  \sum^N_{j=1} w_{i j} v_j[1].
\end{align*}
Hence, $v[2] > v[1]$ is equivalent to
\begin{align}\label{ooopipq}
 \sum^N_{j=1} w_{i j} v_j[1] >  q_i \frac{v_i[1]}{1 - v_i[1]}, \quad i=1, ..., N.
\end{align}
With the geometric series, we can write
\begin{align*}
\frac{v_i[1]}{1 - v_i[1]} = \sum^\infty_{l = 1} v^l_i[1],
\end{align*}
which yields that (\ref{ooopipq}) is equivalent to
\begin{align}\label{kljhnkjhoisss}
 \sum^N_{j=1} w_{i j} v_j[1] >  q_i\sum^\infty_{l = 1} v^l_i[1], \quad i=1, ..., N.
\end{align}
We stack (\ref{kljhnkjhoisss}) and obtain 
\begin{align}\label{ksssaaaljhnkjhoisss}
 W v[1] >  \textrm{diag}(q)\sum^\infty_{l = 1} v^l[1],
\end{align}
where we denote $v^l[1] = (v^l_1[1], ..., v^l_N[1])^T$. By subtracting $\textrm{diag}(q)v[1]$ on both sides of (\ref{ksssaaaljhnkjhoisss}), we obtain the first statement of Theorem \ref{theorem:monotonicity}.

\subsection{Statement 2} \label{subsec:statement2}
We obtain the second statement of Theorem \ref{theorem:monotonicity} by considering when $\Delta v[2] > \Delta v[1]$ holds, which is equivalent to $v[2]>v[1]$. With Proposition \ref{proposition:delta_v_equation}, it holds for node $i$ that 
\begin{equation*}
\Delta v_i[2] - \Delta v_i[1] = \frac{q_i}{ v_{\infty, i} - 1} \Delta v_i[1]  + (1- v_{\infty, i}) \sum^N_{j=1} w_{ij}\Delta v_j[1]-  \Delta v_i[1]  \sum^N_{j=1} w_{ij}\Delta v_j[1].
\end{equation*}
Thus, $\Delta v[2] > \Delta v[1]$ holds if and only if 
\begin{align} \label{kjbnkssssjbksdss} 
 (1- v_{\infty, i}-\Delta v_i[1] ) \sum^N_{j=1} w_{ij}\Delta v_j[1]  > \frac{q_i}{ 1- v_{\infty, i}} \Delta v_i[1]
\end{align}
for every node $i=1,..., N$. The inequality (\ref{kjbnkssssjbksdss}) is equivalent to
\begin{align}  \label{kjbnkjbksdss}
  \sum^N_{j=1} w_{ij}\Delta v_j[1]  >  \frac{q_i}{ 1- v_{\infty, i}} \frac{\Delta v_i[1]}{1- v_{\infty, i}-\Delta v_i[1] }. 
\end{align}
We rewrite the right-hand side of (\ref{kjbnkjbksdss}) to obtain
\begin{align} 
  \sum^N_{j=1} w_{ij}\Delta v_j[1]  > \frac{q_i}{ 1- v_{\infty, i}} \dfrac{\dfrac{\Delta v_i[1]}{ 1- v_{\infty, i}}}{1-\dfrac{\Delta v_i[1]}{ 1- v_{\infty, i}} }  =  \frac{q_i}{1- v_{\infty, i}} \sum^\infty_{l=1} \left(\frac{\Delta v_i[1]}{ 1- v_{\infty, i}}\right)^l, \label{kkkkjjjaaauuuisisisi}
  \end{align}
  where the equality follows from rewriting the geometric series. We introduce $z_i = \Delta v_i[1]/\left( 1- v_{\infty, i}\right)$ for every node $i$, and we obtain from (\ref{kkkkjjjaaauuuisisisi}) that $v[k+1] > v[k]$ is equivalent to 
  \begin{align*}  
  \sum^N_{j=1} w_{ij} \left(1- v_{\infty, j}\right) z_j > \frac{q_i}{1- v_{\infty, i}} \sum^\infty_{l=1} z^l_i, \quad  i=1, ..., N.
  \end{align*}
  We bring the first-order terms on the left-hand side to obtain the equivalent statement
  \begin{align}  \label{kjbkkkssa}
 \left(1- v_{\infty, i}\right)  \sum^N_{j=1} w_{ij} \left(1- v_{\infty, j}\right) z_j -q_i z_i  > q_i   \sum^\infty_{l=2} z^l_i,
  \end{align}
  where $i=1, ..., N$. Stacking (\ref{kjbkkkssa}) yields
  \begin{equation*}  
 \left(\textrm{diag}\left(u - v_{\infty}\right) W  \textrm{diag}\left(u - v_{\infty}\right) -\textrm{diag}(q)\right) z  >  \textrm{diag}(q) \sum^\infty_{l=2} z^l ,
  \end{equation*}
  which completes the proof of the second statement of Theorem \ref{theorem:monotonicity}.
  
\section{Proof of Corollary 2}
We prove Corollary \ref{corollary_monotonicity_12} for the two different initial viral states $v[1]$ in Subsection \ref{subsec:cor_statement1} and Subsection \ref{subsec:cor_statement2}, respectively.

\subsection{First Statement} \label{subsec:cor_statement1}
The initial state is given by
\begin{align*} 
v[1] = \epsilon x_1 + \eta,
\end{align*}  
where the $N\times 1$ vector $\eta$ satisfies $\lVert \eta \rVert_2 = \mathcal{O}(\epsilon^p)$ with $p>1$. By the definition of the principal eigenvector $x_1$, we obtain that
\begin{align*}
 R v[1] = \rho(R) \epsilon x_1 + R\eta.
\end{align*}
Thus, we obtain that
\begin{align} \label{kljnkjdddddd}
 R v[1] = \rho(R) v[1] + \left( R -  \rho(R) I\right)\eta.
\end{align}
We add $v[1]$ on both sides of the inequality of the first statement of Theorem \ref{theorem:monotonicity}, which yields that the viral state $v[k]$ is globally strictly increasing if and only if 
\begin{align*}
 \left( I + W - \textup{\textrm{diag}}(q)\right) v[1] > v[1] +  \textup{\textrm{diag}}(q)\sum^\infty_{l = 2} v^l[1], 
\end{align*}
which simplifies to 
\begin{align}\label{khjbhbia}
R v[1] > v[1] +  \textup{\textrm{diag}}(q)\sum^\infty_{l = 2} v^l[1].
\end{align}
With (\ref{kljnkjdddddd}), we obtain from (\ref{khjbhbia}) that the viral state $v[k]$ is globally strictly increasing if
\begin{align*}
\rho(R) v[1] + \left( R -  \rho(R) I\right)\eta > v[1] +  \textup{\textrm{diag}}(q)\sum^\infty_{l = 2} v^l[1],
\end{align*}
which is equivalent to
\begin{align}\label{inequilknlss}
\left(\rho(R) -1\right) v[1] > \left( \rho(R) I - R\right)\eta + \textup{\textrm{diag}}(q)\sum^\infty_{l = 2} v^l[1].
\end{align}
Since $\rho(R) >1$ and $v[1]>0$, the left-hand side of (\ref{inequilknlss}) is positive and in $\mathcal{O}(\epsilon)$, and the right-hand side of (\ref{inequilknlss}) is in $\mathcal{O}(\epsilon^p)$ with $p>1$. Hence, there is an $\epsilon>0$ such that (\ref{inequilknlss}) holds true.

\subsection{Second Statement} \label{subsec:cor_statement2}
The initial state is given by
\begin{align} \label{hhsaia}
v[1] = (1-\epsilon)v_{\infty} + \eta,
\end{align}  
where the $N\times 1$ vector $\eta$ satisfies $\lVert \eta \rVert_2 = \mathcal{O}(\epsilon^p)$ with $p>1$. With (\ref{hhsaia}), we obtain the $i$-th component of the vector $z$ in Theorem \ref{theorem:monotonicity} as
\begin{align} \label{kjbnmkjnbddd}
z_i = \frac{-\epsilon v_{\infty, i} + \eta_i }{1 - v_{\infty, i}}.
\end{align}
Then, with (\ref{kjbnmkjnbddd}), the inequality in the second statement of Theorem \ref{theorem:monotonicity} becomes
\begin{equation*}
\left( 1 - v_{\infty, i} \right) \sum^N_{j=1} w_{ij} \left(-\epsilon v_{\infty, j} + \eta_j\right)- q_i \frac{-\epsilon v_{\infty, i} + \eta_i }{1 - v_{\infty, i}} > q_i\sum^\infty_{l = 2} z^l_i
\end{equation*}
for every node $i=1, ..., N$. We rearrange and obtain
  \begin{equation}    \label{ppoqqyue}  
-\epsilon \left( 1 - v_{\infty, i} \right) \sum^N_{j=1} w_{ij}  v_{\infty, j}+ \epsilon q_i \frac{ v_{\infty, i} }{1 - v_{\infty, i}} > -\left( 1 - v_{\infty, i} \right) \sum^N_{j=1} w_{ij}  \eta_j+ q_i \frac{\eta_i }{1 - v_{\infty, i}} + q_i\sum^\infty_{l = 2} z^l_i, 
  \end{equation}  
for every node $i$. We rewrite the sum on the left-hand side of (\ref{ppoqqyue}) by using the steady-state equation (\ref{steady_state_componentwise}), which yields
  \begin{equation*}    
-\epsilon q_i v_{\infty, i} + \epsilon q_i \frac{ v_{\infty, i} }{1 - v_{\infty, i}} > -\left( 1 - v_{\infty, i} \right) \sum^N_{j=1} w_{ij}  \eta_j+ q_i \frac{\eta_i }{1 - v_{\infty, i}} + q_i\sum^\infty_{l = 2} z^l_i, 
  \end{equation*}  
  which simplifies to
  \begin{equation}    \label{kjlnkjnnssddd}
\epsilon q_i \frac{ v^2_{\infty, i} }{1 - v_{\infty, i}} > -\left( 1 - v_{\infty, i} \right) \sum^N_{j=1} w_{ij}  \eta_j+ q_i \frac{\eta_i }{1 - v_{\infty, i}} + q_i\sum^\infty_{l = 2} z^l_i.
  \end{equation}
  The left-hand side of (\ref{kjlnkjnnssddd}) is positive and in $\mathcal{O}(\epsilon)$, and the right-hand side of (\ref{kjlnkjnnssddd}) is in $\mathcal{O}(\epsilon^p)$. Hence, there is an $\epsilon >0$ such that the inequality (\ref{kjlnkjnnssddd}) holds true.

\section{Proof of Theorem 4}
Before giving a rigorous proof of the statement of Theorem \ref{theorem:soft_monotonicity}, we give an intuitive explanation. If the initial viral state $v[1]$ is close to zero, then the NIMFA model (\ref{NIMFA_disc_stacked}) is accurately described by its linearisation (\ref{eq:LTI_sys_NIMFA_zero}) around the origin. The viral state $v[k]$ of the LTI system (\ref{eq:LTI_sys_NIMFA_zero}) converges quickly to the principal eigenvector $x_1$ of the system matrix $R$. If the viral state $v[k^*]$ at some time $k^*\ge 1$ is small and almost parallel to the principal eigenvector $x_1$, then it follows from Corollary \ref{corollary_monotonicity_12} that $v[k]$ is strictly increasing at every time $k \ge k^*$.
 
If existent, we denote by $k^*$ the time from which on the viral state $v[k]$ is strictly increasing, i.e. $v[k+1]>v[k]$ for every time $k\ge k^*$. To find an expression for the time $k^*$, we obtain from Theorem \ref{theorem:monotonicity} that the viral state $v[k]$ is increasing at every time $k\ge k^*$ if and only if
\begin{align} \label{ssjhbuuu}
 R v[k^*] > v[k^*] +  \textup{\textrm{diag}}(q)\sum^\infty_{l = 2} v^l[k^*],
\end{align}  
 which follows from adding the viral state $v[k^*]$ at time $k^*$ on both sides of (\ref{monotonicity_first_sta}). We obtain an approximation the viral state $v[k^*]$ at time $k^*$ from the linearisation (\ref{eq:LTI_sys_NIMFA_zero}) of the NIMFA model (\ref{NIMFA_disc_stacked}) around the origin. First, we decompose the matrix $R$ into two addends
\begin{align*}
R = \rho(R)x_1 x^T_1 + B.
\end{align*}
Here, the $N \times N$ matrix $B$ is given by
\begin{align*}
B = R - \rho(R)x_1 x^T_1,
\end{align*}
and it holds that $B x_1 = 0$. Then, the linearisation (\ref{eq:LTI_sys_NIMFA_zero}) yields
\begin{align} \label{linear_one_step}
v[k+1] \approx  R v[k] = \rho(R)x_1 x^T_1 v[k] + B v[k] .
\end{align}
After iterating (\ref{linear_one_step}), the viral state $v[k^*]$ at time $k^*\ge 1$ follows as
\begin{align} \label{kjbnssss}
v[k^*] = R^{k^* - 1} v[1] + \eta[k^*],
\end{align}
where the linearisation error vector $\eta[k^*]$ is in $\lVert \eta[k^*] \rVert_2 = \mathcal{O}\left(\lVert v[1] \rVert^2_2\right)$ for any fixed time $k^*$ when $v[1] \rightarrow 0$. We rewrite (\ref{kjbnssss}) as
\begin{align} \label{asddddaww}
v[k^*] =  \rho(R)^{ k^* - 1 } \left(x^T_1 v[1]\right) x_1 + B^{ k^* - 1 }v[1] + \eta[k^*].
\end{align}
By inserting (\ref{asddddaww}) in (\ref{ssjhbuuu}), we obtain that the viral state $v[k]$ is strictly increasing at every time $k \ge k^*$ if 
\begin{equation} \label{ljnssssaa}
\rho(R)^{ k^*  } x_1 \left(x^T_1 v[1]\right) + B^{ k^* }v[1]> \rho(R)^{ k^* - 1 } x_1 \left(x^T_1 v[1]\right) + B^{ k^* - 1 }v[1] +  \Upsilon[k^*].
\end{equation}
Here, the terms that are non-linear with respect to the initial viral state $v[1]$ are contained in the $N\times 1$ vector $\Upsilon[k^*]$, which equals
\begin{align*}
\Upsilon[k^*] = \textup{\textrm{diag}}(q)\sum^\infty_{l = 2} v^l[k^*] + (I - R) \eta[k^*].
\end{align*}
It holds that $\lVert \Upsilon[k^*] \rVert_2 = \mathcal{O}\left(\lVert v[1] \rVert^2_2\right)$ for any fixed time $k^*$ when $v[1] \rightarrow 0$. We rearrange (\ref{ljnssssaa}), which yields that
\begin{equation} \label{khbsssss}
   \rho(R)^{ k^* - 1 } \left( \rho(R) - 1 \right)\left(x^T_1 v[1]\right) x_1> B^{ k^* - 1 } \left( I - B \right)v[1] + \Upsilon[k^*].
\end{equation}
To obtain a bound on the time $k^*$ from (\ref{khbsssss}), we state Lemma \ref{lemma:left_side} and Lemma \ref{lemma:right_side}, which bound the left and right side of (\ref{khbsssss}), respectively.
\begin{lemma}\label{lemma:left_side}
Suppose that Assumption \ref{assumption:spreading_parameters}--\ref{assumption:above_threshold} hold. Then, it holds that
\begin{equation*} 
\left(\rho(R)^{ k^* - 1 } \left( \rho(R) - 1 \right)\left(x^T_1 v[1]\right) x_1 \right)_i \ge \rho(R)^{ k^* - 1 } \left( \rho(R) - 1 \right)x^2_{1, \textup{\textrm{min}}} \lVert v[1]\rVert_1
\end{equation*}
 for every $i=1, ..., N$, where $x_{1, \textup{\textrm{min}}} = \textup{\textrm{min}}\{\left(x_1\right)_1, ..., \left(x_1\right)_N\}$.
\end{lemma}
\begin{proof}
 It holds
\begin{align} \label{ljnkljnkjnksdsss}
x^T_1 v[1] \ge  x_{1, \textrm{min}} \sum^N_{i=1}v_i[1]  = x_{1, \textrm{min}} \lVert v[1]\rVert_1,
\end{align}
since $v[1]\ge 0$ and $x_{1, \textrm{min}}>0$, since the principal eigenvector satisfies $x_1 >0$ by Lemma \ref{lemma:R_irreducible}. With (\ref{ljnkljnkjnksdsss}), the $i$-th component of the left-hand side of (\ref{khbsssss}) becomes 
\begin{equation*} 
\left(\rho(R)^{ k^* - 1 } \left( \rho(R) - 1 \right)\left(x^T_1 v[1]\right) x_1 \right)_i \ge \rho(R)^{ k^* - 1 } \left( \rho(R) - 1 \right)x_{1, \textrm{min}} \lVert v[1]\rVert_1 \left( x_1 \right)_i.
\end{equation*}
since $\left( \rho(R) - 1 \right) > 0$. By employing the lower bound $\left( x_1 \right)_i  \ge x_{1, \textrm{min}}$, we have proved Lemma \ref{lemma:left_side}.
\end{proof}
For completeness, we introduce Lemma \ref{lemma:matrix_norm_spectral_radius}, which is from \cite[Corollary 5.6.13.]{horn1990matrix} and applied in the proof of Lemma \ref{lemma:right_side}.
\begin{lemma}[\cite{horn1990matrix}]\label{lemma:matrix_norm_spectral_radius}
Let an $N\times N$ matrix $M$ and an $\varepsilon>0$ be given. Then, there is a constant $C(M, \varepsilon)$ such that 
\begin{align*}
\left(M^k\right)_{ij} \le C(M, \varepsilon) \left( \rho(M) + \varepsilon \right)^k
\end{align*}
for all $k=1, 2, ...$ and all $i,j=1, ..., N$.
\end{lemma}
For any $N \times 1$ vector $z$, the maximum vector norm is given by
\begin{align*}
 \lVert z \rVert_\infty = \textrm{max}\{ |z_1|, ..., |z_N| \}.
\end{align*}
For any $N\times N$ matrix $M$ with elements $m_{ij}$, we denote the matrix norm which is induced the maximum vector norm by
\begin{align}\label{matrix_max_norm}
\lVert M \rVert_\infty = \underset{i=1, ..., N}{\textrm{max}} ~\sum^N_{j=1} |m_{ij}|.
\end{align}
\begin{lemma} \label{lemma:right_side}
Suppose that Assumption \ref{assumption:spreading_parameters}--\ref{assumption:above_threshold} hold, and let $\varepsilon >0$ be given. Then, there is a constant $C(B, \varepsilon)$ such that
\begin{equation*}
 \left(B^{ k^* - 1 } \left( I - B \right)v[1] \right)_i \le  C(B, \varepsilon) \left( \rho(B) + \varepsilon \right)^{k^*-1}  \left\lVert I - B \right\rVert_\infty ~ \lVert v[1] \rVert_1
\end{equation*}
holds for every integer $k^* \ge 2$ and every $i=1, ..., N$.
\end{lemma}
\begin{proof}
For any $N\times 1$ vector $z$ and any $N \times N$ matrix $M$, it holds
\begin{align*}
 \left(M z\right)_i = \sum^N_{j=1} m_{ij} z_j  \le \sum^N_{l=1} |m_{il}| \sum^N_{j=1} |z_j|.
\end{align*}
From (\ref{matrix_max_norm}) and $\lVert z \rVert_1 = \sum^N_{j=1} |z_j|$, we obtain that
\begin{align}\label{kljnkjbnisss}
(M z)_i \le \lVert M \rVert_\infty ~ \lVert z \rVert_1 , \quad i=1, ..., N,
\end{align}
for any vector $z$ and any square matrix $M$. By setting the matrix $M$ to $M= \left(B^{ k^* - 1 } \left( I - B \right)\right)$ and the vector $z$ to $z = v[1]$, we obtain from (\ref{kljnkjbnisss}) that
\begin{align*}
\left(B^{ k^* - 1 } \left( I - B \right)v[1] \right)_i \le \left\lVert B^{ k^* - 1 } \left( I - B \right) \right\rVert_\infty ~ \lVert v[1] \rVert_1
\end{align*}
for every $i=1, ..., N$. Since the matrix norm is sub-multiplicative\footnote{A matrix norm $\lVert \cdot \rVert$ is sub-multiplicative if $\lVert A B\rVert \le\lVert A\rVert \lVert B\rVert$ holds for any matrices $A, B$.}, it holds that
\begin{align} \label{kljnkhjbnaa}
\left(B^{ k^* - 1 } \left( I - B \right)v[1] \right)_i \le \left\lVert B^{ k^* - 1 } \right\rVert_\infty ~ \left\lVert I -B \right\rVert_\infty ~\lVert v[1] \rVert_1.
\end{align}
For a given matrix $M$ and a given $\varepsilon>0$, there is a constant $C(M, \varepsilon)$ such that 
\begin{align} \label{bounddddasss}
\lVert M^k \rVert_\infty \le C(M, \varepsilon) \left( \rho(M) + \varepsilon \right)^k
\end{align}
for all integers $k\ge 1$, which follows from Lemma \ref{lemma:matrix_norm_spectral_radius}. We combine (\ref{bounddddasss}) and (\ref{kljnkhjbnaa}) and obtain that, for any $\varepsilon>0$, there is a constant $C(B, \varepsilon)$ such that 
\begin{equation*} 
\left(B^{ k^* - 1 } \left( I - B \right)v[1] \right)_i \le  C(B, \varepsilon) \left( \rho(B) + \varepsilon \right)^{k^*-1} \left\lVert I - B \right\rVert_\infty ~ \lVert v[1] \rVert_1
\end{equation*}
holds for every integer $k^* \ge 2$ and every node $i=1, ..., N$.  
\end{proof}
By applying the bounds of Lemma \ref{lemma:left_side} and Lemma \ref{lemma:right_side} to (\ref{khbsssss}), we obtain that the viral state $v[k]$ is strictly increasing at every time $k\ge k^*$ if
\begin{equation} \label{inequaaajnkjs}
\rho(R)^{ k^* - 1 } \left( \rho(R) - 1 \right)x^2_{1, \textup{\textrm{min}}} \lVert v[1]\rVert_1 > C(B, \varepsilon) \left( \rho(B) + \varepsilon \right)^{k^*-1} \left\lVert I - B \right\rVert_\infty ~ \lVert v[1] \rVert_1  + \Upsilon[k^*].
\end{equation}
In the limit $v[1]\rightarrow 0$, it holds $\Upsilon[k^*] = \mathcal{O}(\lVert v[1]\rVert^2_2)$ for $k^*$ fixed, and the inequality (\ref{inequaaajnkjs}) converges to
\begin{equation*} 
\rho(R)^{ k^* - 1 } \left( \rho(R) - 1 \right)x^2_{1, \textup{\textrm{min}}} > C(B, \varepsilon) \left( \rho(B) + \varepsilon \right)^{k^*-1} \left\lVert I - B \right\rVert_\infty.
\end{equation*}
We take the logarithm and obtain
\begin{equation} \label{kljbniouhf}
 \log\left( \left( \rho(R) - 1 \right)x^2_{1, \textup{\textrm{min}}} \right) > \log\left( C(B, \varepsilon)\left\lVert I - B \right\rVert_\infty \right) + \left(k^*-1\right)\log\left(\frac{\rho(B) + \varepsilon}{\rho(R)} \right).
 \end{equation}
We choose $\varepsilon$ such that $\rho(B)+ \varepsilon < \rho(R)$ and find that (\ref{kljbniouhf}) is satisfied if
\begin{equation} \label{kjbkassss}
 k^* > \frac{ \log\left( \dfrac{\left( \rho(R) - 1 \right)x^2_{1, \textrm{min}}}{ C(B, \varepsilon)\left\lVert I - B \right\rVert_\infty} \right)}{\log\left(\rho(B) + \varepsilon \right) - \log\left(\rho(R) \right)} + 1.
\end{equation}
Hence, in the limit $v[1] \rightarrow 0$, the viral state $v[k]$ is strictly increasing at every time $k\ge k^*$ if $k^*$ satisfies (\ref{kjbkassss}), and we emphasise that (\ref{kjbkassss}) is independent of $v[1]$. Thus, when $v[1] \rightarrow 0$, the set $S_-$ of time instants $k$, for which the viral state $v[k]$ is not strictly increasing, is a subset of $\{1, ..., k^* - 1\}$. Hence, the set $S_-$ is finite when $v[1] \rightarrow 0$, which is the first requirement for a quasi-increasing viral state evolution by Definition \ref{definition:quasi_increasing}. It remains to show that, for any $\epsilon$-stringency, 
\begin{align} \label{jkbkooaaass}
\lVert v[k+1] - v[k] \rVert_2 \le \epsilon \quad \forall k \in S_-,
\end{align}
if $\lVert v[1] \rVert_2 \le \vartheta(\epsilon)$ for a sufficiently small $\vartheta(\epsilon)$. With the triangle inequality it holds that
\begin{align*}
\lVert v[k+1] - v[k] \rVert_2 \le \lVert v[k+1] \rVert_2 + \lVert v[k] \rVert_2, \quad \forall k \in S_-.
\end{align*}
Since $v[1] \rightarrow 0$ implies that $v[k] \rightarrow 0$ for every time $k \le k^* +1$, we obtain that, for any $\epsilon$-stringency, there is a $\vartheta(\epsilon)$ such that $\lVert v[1] \rVert_2 \le \vartheta(\epsilon)$ implies (\ref{jkbkooaaass}).

\section{Proof of Proposition 2}
From (\ref{NIMFA_disc_}) and (\ref{lti_upper_bound_A}) follows that 
\begin{equation} \label{kjbksssqqq}
v^{(1)}_{\textrm{ub}, i} [k + 1] -v_i [k + 1] = (1 - q_i) \left( v^{(1)}_{\textrm{ub}, i}[k] - v_i[k]\right) +  \sum^N_{j=1} w_{i j} \left( v^{(1)}_{\textrm{ub}, j}[k] - v_j[k]\right)+ v_i[k] \sum^N_{j=1} w_{i j} v_j[k].
\end{equation}
Hence, $v^{(1)}_{\textrm{ub}, j}[k] \ge v_j[k]$ for every node $j$ implies that the first term and the first sum of (\ref{kjbksssqqq}) are non-negative. Since the second sum in (\ref{kjbksssqqq}) is positive, it follows from (\ref{kjbksssqqq}) that $v^{(1)}_{\textrm{ub}, i} [k + 1] > v_i [k + 1]$ if $v^{(1)}_{\textrm{ub}, j}[k] \ge v_j[k]$ for every node $j$. The LTI system (\ref{lti_upper_bound_A}) is asymptotically stable if and only if the spectral radius $\rho(R)$ satisfies $\rho(R) < 1$.

\section{Proof of Proposition 3}
\subsection{Statement 1}
 We prove that $\Delta v_{\textrm{ub},i}[k]$, given by (\ref{lti_upper_bound_B}), is indeed an upper bound of $\Delta v_i[k]$ for all nodes $i$ at every time $k\ge 1$ by induction. For the initial time $k=1$, it holds $\Delta v_{\textrm{ub},i}[1] \ge \Delta v_i[1]$ by assumption. In the following, we show that $\Delta v_{\textrm{ub}, i}[k] \ge  \Delta v_i[k]$ for all nodes $i$ implies $\Delta v_{\textrm{ub}, i}[k+1] \ge  \Delta v_i[k+1]$ for all nodes $i$. From (\ref{lti_upper_bound_B}) and (\ref{delta_v_system}) follows that the difference of the bound $\Delta v_{\textrm{ub}, i}[k+1]$ to the true value $\Delta v_{i}[k+1]$ at time $k+1$ can be stated as
\begin{equation} \label{kljbnkljbn}
\Delta v_{\textrm{ub}}[k + 1] - \Delta v[k + 1] = F \left(\Delta v_{\textrm{ub}}[k] - \Delta v[k]  \right) + \textrm{diag}\left(\Delta v[k]\right) W \Delta v[k].
\end{equation}
For the first addend in (\ref{kljbnkljbn}) it holds that
\begin{equation*}
F \left(\Delta v_{\textrm{ub}}[k] - \Delta v[k]  \right)  \ge 0,
\end{equation*}
since $\Delta v_{\textrm{ub}}[k] - \Delta v[k] \ge 0$ and since the matrix $F$ is non-negative by Lemma \ref{lemma:nonnegativity}. Under Assumption \ref{assumption:initial_state}, Corollary \ref{corollary:below_steady_state} implies that $\Delta v_i[k] \le 0$ for every node $i$ at every time $k\ge 1$. Thus, we obtain for the second addend in (\ref{kljbnkljbn}) that
\begin{equation*}
 \sum^N_{j=1} w_{ij} \Delta v_j[k] \Delta v_i[k] \ge 0, \quad i=1, ..., N,
\end{equation*}
since $w_{ij}\ge 0$ for every $i,j=1, ..., N$ under Assumption \ref{assumption:spreading_parameters}. Thus, both addends of (\ref{kljbnkljbn}) are non-negative, which implies that $\Delta v_{\textrm{ub}}[k + 1] \ge \Delta v[k + 1]$.
\subsection{Statement 2}
Under Assumption \ref{assumption:sampling_time}, the matrix $F$ is non-negative as stated by Lemma \ref{lemma:nonnegativity}. Hence, we obtain from
\begin{equation*}
\Delta v_i[k+1] = \sum^N_{j=1}F_{ij}\Delta v_i[k]
\end{equation*}
that $\Delta v_i[k] \le 0$ for every node $i$ implies $\Delta v_i[k+1] \le 0$ for every node $i$.

\section{Proof of Proposition 4}
\subsection{Statement 1}
Since  $F_\textrm{lb}  =F - \textrm{diag}\left(v_\textrm{min} -v_\infty \right)W$, we can rewrite the lower bound $\Delta v_\textrm{lb}[ k + 1 ]$ at time $k+1$ with (\ref{lti_lower_bound}) as 
\begin{equation}\label{dkjbiias}
\Delta v_\textrm{lb}[ k + 1 ] = F \Delta v_\textrm{lb}[ k ]- \textrm{diag}\left(v_\textrm{min} -v_\infty \right) W \Delta v_\textrm{lb}[ k ].
\end{equation}
We prove that $\Delta v_\textrm{lb}[k]$ given by (\ref{dkjbiias}) is indeed a lower bound of $\Delta v[k]$ at every time $k\ge 1$ by induction. For the initial time $k=1$, it holds $\Delta v_\textrm{lb}[1] \le \Delta v[1]$ by assumption. In the following, we show that $\Delta v_\textrm{lb}[k] \le  \Delta v[k]$ implies $\Delta v_\textrm{lb}[k+1] \le  \Delta v[k+1]$ for any time $k$. We obtain from (\ref{delta_v_system}) and (\ref{dkjbiias}) that
\begin{equation}\label{ljnkljdsss}
\Delta v[ k + 1 ] - \Delta v_\textrm{lb}[ k + 1 ] = F\left(\Delta v[ k ] - \Delta v_\textrm{lb}[ k  ]\right) + \textrm{diag}\left(v_\textrm{min} -v_\infty \right) W \Delta v_\textrm{lb}[ k ] - \textrm{diag}\left( \Delta v[k]\right) W \Delta v[k] .
\end{equation}
Under Assumption \ref{assumption:sampling_time}, the matrix $F$ is non-negative as stated by Lemma \ref{lemma:nonnegativity}. From the non-negativity of the matrix $F$ and from $\Delta v[ k ] \ge \Delta v_\textrm{lb}[ k  ]$ follows that the first term of (\ref{dkjbiias}) is non-negative, i.e. 
\begin{equation*}
F(\Delta v[ k ] - \Delta v_\textrm{lb}[ k  ]) \ge 0.
\end{equation*}
We denote the $i$-th component of second term in (\ref{ljnkljdsss}) by
\begin{equation*}
\varsigma_i = \sum^N_{j  = 1} w_{ij} \left( \left(v_{\textrm{min}, i} -v_{\infty, i} \right)  \Delta v_{\textrm{lb}, j}[ k ] - \Delta v_i[k] \Delta v_j[k] \right).
\end{equation*}
Under Assumption \ref{assumption:initial_state} it holds $\Delta v_i[k] \le 0$ as stated by Corollary \ref{corollary:below_steady_state}. Furthermore, since $\Delta v_j[k] \ge \Delta v_{\textrm{lb}, j}[ k ]$, we obtain that
\begin{equation*}
\varsigma_i \ge \sum^N_{j  = 1} w_{ij} \left( \left(v_{\textrm{min}, i} -v_{\infty, i} \right) \Delta v_{\textrm{lb}, j}[ k ] - \Delta v_i[k] \Delta v_{\textrm{lb}, j}[ k ] \right). 
\end{equation*}
Since we assumed that $v[k]\ge v_\textrm{min}$ holds at every time $k$, we obtain that $\Delta v[k] \ge v_\textrm{min} - v_{\infty}$ at every time $k$. Hence, we can lower bound the term $\varsigma_i$ by
\begin{equation*}
\varsigma_i \ge \sum^N_{j  = 1} w_{ij} ( \left(v_{\textrm{min}, i} -v_{\infty, i} \right)  \Delta v_{\textrm{lb}, j}[ k ] -  \left(v_{\textrm{min}, i} -v_{\infty, i} \right)\Delta v_{\textrm{lb}, j}[ k ] )= 0.
\end{equation*}
Thus, both terms in (\ref{ljnkljdsss}) are non-negative, which implies that $\Delta v[ k + 1 ] \ge \Delta v_\textrm{lb}[ k + 1 ] $ if $\Delta v[ k ] \ge \Delta v_\textrm{lb}[ k ]$.

\subsection{Statement 2} \label{subsec:lower_bound_statement2}
 The proof is in parts inspired by the proof of Ahn \textit{et al.} \cite[Theorem 5.1]{ahn2013global} and based on two lemmas.
\begin{lemma}\label{lemma:F_lb_pos}
For any two vectors $z, \tilde{z}$ with $z \ge \tilde{z}$ it holds that $F_\textup{\textrm{lb}} z \ge F_\textup{\textrm{lb}} \tilde{z}$.
\end{lemma}
\begin{proof}
First, we show that the matrix $F_\textrm{lb}$ is non-negative. The elements of the matrix $F_\textrm{lb}$ are given by
\begin{align}\label{kjbkjboaaasss}
\left( F_\textrm{lb}\right)_{ij} = \begin{cases}
1 + \dfrac{q_i}{v_{\infty, i} - 1} + \left( 1- v_{\textrm{min}, i}\right) w_{ii}  \quad &\text{if $i=j$}, \\
\left( 1- v_{\textrm{min}, i}\right) w_{ij} &\text{if $i\neq j$}.
\end{cases}
\end{align}
For every node $i$, we have $\left( F_\textrm{lb}\right)_{ii} \ge F_{ii} \ge 0$ under Assumption \ref{assumption:sampling_time} as stated by Lemma \ref{lemma:nonnegativity}. Since $v_{\textrm{min}, i} <1$ and $w_{ij} \ge 0$ for every nodes $i, j$, the matrix $F_\textrm{lb}$ is non-negative. Hence, $z \ge \tilde{z}$ implies that
\begin{equation*}
\left( F_\textrm{lb}z - F_\textrm{lb} \tilde{z} \right)_i = \sum^N_{j=1} \left( F_\textrm{lb}\right)_{ij} \left(z_j - \tilde{z}_j \right) \ge 0, \quad \forall i=1, ..., N.
\end{equation*}
\end{proof}
\begin{lemma}\label{lemma:z_k_lower_bound}
Define the $N\times 1$ vector $z^{(0)}$ as
\begin{equation*}
z^{(0)} = -  v_\infty
\end{equation*}
and the $N\times 1$ vectors $z^{(k+1)}$ as
\begin{equation*}
z^{(k+1)} = F_\textup{\textrm{lb}} z^{(k)} \quad k\ge 0.
\end{equation*}
Then, the vector $z^{(k)}$ at iteration $k$ can be lower bounded by 
\begin{equation} \label{inequuuuuu}
z^{(k)} \ge -  \left( 1  - q_\textup{\textrm{min}} \frac{\gamma}{1- \gamma}  \right)^k v_{\infty}.
\end{equation}
\end{lemma}
\begin{proof}
The right-hand side of (\ref{inequuuuuu}) is parallel to the steady-state vector $v_\infty$, and, as a first step, we consider the product $\left(- F_\textup{\textrm{lb}} v_\infty\right)$. With (\ref{kjbkjboaaasss}), we obtain for every $i=1, ..., N$ that
\begin{align*}
 \left(- F_\textup{\textrm{lb}} v_\infty\right)_i = -\sum^N_{j=1} \left( F_\textrm{lb} \right)_{ij} v_{\infty, j} =  -v_{\infty, i} +  q_i \frac{ v_{\infty, i}}{ 1 -v_{\infty, i}}  -  \left(1- v_{\textrm{min}, i} \right) \sum^N_{j=1} w_{ij} v_{\infty, j}.
\end{align*}
 The steady-state equation (\ref{steady_state_componentwise}) yields that
\begin{equation}\label{kjbkbkasss}
 -\sum^N_{j=1} \left( F_\textrm{lb} \right)_{ij} v_{\infty, j}  =  -v_{\infty, i} +  q_i \frac{ v_{\infty, i}}{ 1 -v_{\infty, i}} -  \left(1- v_{\textrm{min}, i} \right) q_i \frac{v_{\infty, i}}{1 - v_{\infty, i}}.
\end{equation}
We simplify (\ref{kjbkbkasss}) and obtain
\begin{equation}\label{kjniossss}
-\sum^N_{j=1} \left( F_\textrm{lb} \right)_{ij} v_{\infty, j}  = -v_{\infty, i} \left(  1 - q_i \frac{v_{\textrm{min}, i}}{1 - v_{\infty, i}}\right).
\end{equation}
We stack (\ref{kjniossss}), which yields that
\begin{equation} \label{kjbkjhbiubhas}
 -F_\textrm{lb}  v_\infty \ge -  v_{\infty} \left( 1 - q_\textup{\textrm{min}} \frac{\gamma}{1- \gamma}  \right),
\end{equation}
since $q_i \ge q_\textrm{min}$ and $v_{\infty, i} \ge v_{\textrm{min}, i} \ge \gamma$ for every $i=1, ..., N$. As a second step, we obtain the inequality (\ref{inequuuuuu}) from (\ref{kjbkjhbiubhas}) by induction. At iteration $k=0$, (\ref{inequuuuuu}) holds true with equality. Consider that (\ref{inequuuuuu}) holds at time $k\ge 0$, then we obtain 
\begin{equation*} 
z^{(k+1)} = F_\textrm{lb} z^{(k)} \ge F_\textrm{lb} \left( -  \left( 1  - q_\textup{\textrm{min}} \frac{\gamma}{1- \gamma}  \right)^k v_{\infty}\right),
\end{equation*}
where the inequality follows from Lemma \ref{lemma:F_lb_pos}. Finally, with (\ref{kjbkjhbiubhas}), we obtain that
\begin{equation*} 
z^{(k+1)} \ge  - \left( 1  - q_\textup{\textrm{min}} \frac{\gamma}{1- \gamma}  \right)^{k+1} v_\infty.
\end{equation*}
\end{proof}

Since $\Delta v_\textrm{lb}[1] = \Delta v[1]$ and $\Delta v[1] = v[1]-v_\infty$, it holds that $\Delta v_\textrm{lb}[1] \ge -v_\infty = z^{(0)}$. Hence, Lemma \ref{lemma:F_lb_pos} and Lemma \ref{lemma:z_k_lower_bound} imply (by induction) that 
\begin{equation*}
\Delta v_\textrm{lb}[k] \ge z^{(k-1)} \ge  - \left( 1  - q_\textup{\textrm{min}} \frac{\gamma}{1- \gamma}  \right)^{k-1} v_\infty 
\end{equation*}
at every time $k\ge 1$.

\section{Proof of Corollary 3}
If the viral state $v[k]$ is globally strictly increasing and $v[1]~>~0$, then we can set $v_\textrm{min} = v[1]>0$. Since
\begin{equation*}
\Delta v_\textrm{lb} [k] \le \Delta v[k] \le 0 \Rightarrow \lVert \Delta v[k] \rVert_2 \le \lVert \Delta v_\textrm{lb} [k] \rVert_2,
\end{equation*}
 Proposition \ref{proposition:lower_bound_lti} yields that
\begin{equation*}
 \lVert \Delta v[k] \rVert_2 \le   \left( 1  - q_\textrm{min} \frac{\gamma}{1- \gamma}  \right)^{k-1} \lVert v_\infty \rVert_2 \quad \forall k \ge 1,
\end{equation*} 
 where $\gamma = \textrm{min}\{ v_1[1], ...,  v_N[1]\}>0$. If the viral state is \textit{not} globally strictly increasing, then we first show the existence of an $N \times 1$ vector $v_\textup{\textrm{min}}>0$ before applying Proposition \ref{proposition:lower_bound_lti}.
\begin{lemma} \label{lemma:v_min_existence}
Suppose that Assumptions \ref{assumption:spreading_parameters}--\ref{assumption:above_threshold} hold. Then, for any initial viral state $v[1] > 0$, there is an $N \times 1$ vector $v_\textup{\textrm{min}}>0$ such that $v[k]\ge v_\textup{\textrm{min}}$ holds at every time $k\ge 1$.
\end{lemma}
\begin{proof}
The proof consists of three parts: First, it follows from the NIMFA equations (\ref{NIMFA_disc_}) that $v[k]>0$ implies $v[k+1]>0$. Hence, it holds that $v[k]>0$ at every time $k\ge 1$.

Second, the viral state vector $v[k]$ does not approach zero arbitrarily close: Under Assumption \ref{assumption:above_threshold}, the origin is an unstable equilibrium of the NIMFA equations (\ref{NIMFA_disc_stacked}). From $v_i[k] > 0$ for every node $i$ we obtain that $x_1^T v[k] >0$, where $x_1>0$ is the principal eigenvector to the unstable eigenvalue $\rho(R)>1$ of the linearisation (\ref{eq:LTI_sys_NIMFA_zero}) of the NIMFA model (\ref{NIMFA_disc_stacked}) around the origin. 

Third, the viral state $v_i[k]$ of any single node $i$ does not approach zero arbitrarily close if $v[k]> 0$ at every time $k\ge 1$. (This is a stronger statement than the second statement). Since the viral state vector $v[k]$ does not approach zero arbitrarily close, there is at least one node $i$ such that $v_i[k] \ge v_{\textrm{min}, i}$ for some $v_{\textrm{min}, i} >0$ at every time $k \ge 1$. Due to Assumption \ref{assumption:connected_graph}, node $i$ has at least one neighbour $l \neq i$. Now suppose that for node $l$ the viral state $v_l[k]$ does approach zero arbitrarily close, then the NIMFA equations (\ref{NIMFA_disc_}) become
\begin{equation*}
v_l [k + 1] = \sum^N_{j=1} w_{l j} v_j[k] \ge w_{l i} v_i[k] \ge  w_{l i } v_{\textrm{min}, i}>0, 
\end{equation*}
which is a contradiction to $v_l[k] \rightarrow 0$. Thus, the viral states $v_l[k]$ of node $l$ does not approach zero arbitrarily close. Hence, if there is a $v_{\textrm{min}, i} >0$ for some node $i$ such that $v_i[k] \ge v_{\textrm{min}, i}$ at every time $k\ge 1$, then there is a $v_{\textrm{min}, l} >0$ for some node $l \neq i$ such that $v_l[k] \ge v_{\textrm{min}, l}$ at every time $k\ge 1$. Repeating this argument for every node yields that there is a positive vector $v_\textrm{min}>0$ such that $v[k]\ge v_\textup{\textrm{min}}$ holds at every time $k\ge 1$.
\end{proof}
As stated by Lemma \ref{lemma:v_min_existence}, if the viral state is not globally strictly increasing, then it holds that $v[k]\ge v_\textup{\textrm{min}}$ for some vector $v_\textup{\textrm{min}}>0$ at every time $k\ge 1$, and Proposition \ref{proposition:lower_bound_lti} yields that 
\begin{equation*}
 \lVert \Delta v[k] \rVert_2  \le  \left( 1  - q_\textrm{min} \frac{\gamma}{1- \gamma}  \right)^{k-1} \lVert v_\infty \rVert_2, \quad ~ k\ge 1,
\end{equation*}
where $\gamma = \textrm{min}\{v_{\textrm{min}, 1}, ..., v_{\textrm{min}, N}\}>0$. By setting $\alpha = \left( 1  - q_\textrm{min} \frac{\gamma}{1- \gamma}  \right)$, we obtain Corollary \ref{corollary:exponential_stability}.
 

\begin{thebibliography}{10}
\bibitem{bailey1975mathematical}
N.~T. Bailey \emph{et~al.}, \emph{The mathematical theory of infectious
  diseases and its applications}.\hskip 1em plus 0.5em minus 0.4em\relax
  Charles Griffin \& Company Ltd 5a Crendon Street, High Wycombe, Bucks HP13
  6LE., 1975, no. 2nd ediition.

\bibitem{anderson1992infectious}
R.~M. Anderson and R.~M. May, \emph{Infectious diseases of humans: dynamics and
  control}.\hskip 1em plus 0.5em minus 0.4em\relax Oxford University Press,
  1992.

\bibitem{pastor2015epidemic}
R.~Pastor-Satorras, C.~Castellano, P.~Van~Mieghem, and A.~Vespignani,
  ``Epidemic processes in complex networks,'' \emph{Reviews of {M}odern
  {P}hysics}, vol.~87, no.~3, p. 925, 2015.

\bibitem{nowzari2016analysis}
C.~Nowzari, V.~M. Preciado, and G.~J. Pappas, ``Analysis and control of
  epidemics: A survey of spreading processes on complex networks,'' \emph{IEEE
  Control Systems Magazine}, vol.~36, no.~1, pp. 26--46, 2016.

\bibitem{van2009virus}
P.~Van~Mieghem, J.~Omic, and R.~Kooij, ``Virus spread in networks,''
  \emph{IEEE/ACM Transactions on Networking}, vol.~17, no.~1, pp. 1--14, 2009.

\bibitem{van2011n}
P.~Van~Mieghem, ``The {N-I}ntertwined {SIS} epidemic network model,''
  \emph{Computing}, vol.~93, no. 2-4, pp. 147--169, 2011.

\bibitem{mieghem2014homogeneous}
P.~Mieghem and J.~Omic, ``In-homogeneous virus spread in networks,''
  \emph{arXiv preprint arXiv:1306.2588}, 2014.

\bibitem{stoer2013introduction}
J.~Stoer and R.~Bulirsch, \emph{Introduction to numerical analysis}.\hskip 1em
  plus 0.5em minus 0.4em\relax Springer Science \& Business Media, 2013,
  vol.~12.

\bibitem{fall2007epidemiological}
A.~Fall, A.~Iggidr, G.~Sallet, and J.-J. Tewa, ``Epidemiological models and
  {L}yapunov functions,'' \emph{Mathematical Modelling of Natural Phenomena},
  vol.~2, no.~1, pp. 62--83, 2007.

\bibitem{pare2018analysis}
P.~E. Par{\'e}, J.~Liu, C.~L. Beck, B.~E. Kirwan, and T.~Ba{\c{s}}ar,
  ``Analysis, estimation, and validation of discrete-time epidemic processes,''
  \emph{IEEE Transactions on Control Systems Technology}, 2018.

\bibitem{ahn2013global}
H.~J. Ahn and B.~Hassibi, ``Global dynamics of epidemic spread over complex
  networks,'' in \emph{Decision and Control (CDC), 2013 IEEE 52nd Annual
  Conference on}.\hskip 1em plus 0.5em minus 0.4em\relax IEEE, 2013, pp.
  4579--4585.

\bibitem{khanafer2016stability}
A.~Khanafer, T.~Ba{\c{s}}ar, and B.~Gharesifard, ``Stability of epidemic models
  over directed graphs: A positive systems approach,'' \emph{Automatica},
  vol.~74, pp. 126--134, 2016.

\bibitem{prasse2018networkreconstruction}
B.~Prasse and P.~Van~Mieghem, ``Network reconstruction and prediction of
  epidemic outbreaks for {NIMFA} processes,'' \emph{arXiv preprint
  arXiv:1811.06741}, 2018.

\bibitem{khalil1996nonlinear}
H.~K. Khalil, ``Nonlinear systems,'' \emph{Prentice-Hall, New Jersey}, vol.~2,
  no.~5, pp. 5--1, 1996.

\bibitem{van2016universality}
P.~Van~Mieghem, ``Universality of the {SIS} prevalence in networks,''
  \emph{arXiv preprint arXiv:1612.01386}, 2016.

\bibitem{daley2001epidemic}
D.~J. Daley and J.~Gani, \emph{Epidemic modelling: an introduction}.\hskip 1em
  plus 0.5em minus 0.4em\relax Cambridge University Press, 2001, vol.~15.

\bibitem{he2018spreading}
Z.~He and P.~Van~Mieghem, ``The spreading time in {SIS} epidemics on
  networks,'' \emph{Physica A: Statistical Mechanics and its Applications},
  vol. 494, pp. 317--330, 2018.

\bibitem{van2014time}
R.~van~de Bovenkamp and P.~Van~Mieghem, ``Time to metastable state in {SIS}
  epidemics on graphs,'' in \emph{2014 Tenth International Conference on
  Signal-Image Technology and Internet-Based Systems}.\hskip 1em plus 0.5em
  minus 0.4em\relax IEEE, 2014, pp. 347--354.

\bibitem{horn1990matrix}
R.~A. Horn and C.~R. Johnson, \emph{Matrix analysis}.\hskip 1em plus 0.5em
  minus 0.4em\relax Cambridge University Press, 1990.

\end{thebibliography}
\end{document}